\documentclass[a4paper,12pt]{amsart}
\usepackage[dvipdfmx]{graphicx}
\usepackage{latexsym,amssymb,amsmath,mathrsfs}
\usepackage{amsthm}
\usepackage{mathrsfs}
\usepackage{color}
\usepackage[english]{babel}
\usepackage{comment}
\usepackage[abbrev]{amsrefs}
\usepackage[top=30truemm,bottom=30truemm,left=20truemm,right=20truemm]{geometry}
\newtheorem{theorem}{Theorem}[section]
\newtheorem{lemma}[theorem]{Lemma}
\newtheorem{proposition}[theorem]{Proposition}
\newtheorem{corollary}[theorem]{Corollary}
\theoremstyle{remark}
\newtheorem{remark}[theorem]{Remark}

\theoremstyle{remark}
\newtheorem{example}[theorem]{Example}

\DeclareMathOperator{\Ric}{Ric}

\address{Graduate School of Mathematical Sciences, The University of Tokyo, Komaba, Tokyo, 153-8914, Japan} 
\email{yasuakifujitani@g.ecc.u-tokyo.ac.jp}

\newcommand{\e}{\mathrm{e}}
\newcommand{\Hess}{\mathrm{Hess}\,}
\renewcommand{\d}{\mathrm{d}}

\newcommand{\eps}{{\varepsilon}}
\makeatletter
\@namedef{subjclassname@2020}{\textup{2020} Mathematics Subject Classification}
\makeatother

\title[Comparison theorems in centro-affine differential geometry]{Comparison theorems in \\centro-affine differential geometry}
\author{Yasuaki Fujitani}
\keywords{Affine connection, Weighted Ricci curvature, Comparison geometry}
\subjclass[2020]{53B05, 53A15, 53C21}
\begin{document}
\begin{abstract}
For weighted manifolds with Bakry-\'{E}mery Ricci curvature bounded from below,
comparison geometric results have been established.
The Bakry-\'{E}mery Ricci curvature depends on the effective dimension $N$.
In this paper,
the case $N = 1$ is further studied in the setting of centro-affine differential geometry.
Several examples,
such as rotationally symmetric ellipsoids and hyperboloids,
arise in rigidity phenomena of comparison theorems.
\end{abstract}
\maketitle
\section{Introduction}\label{sec:intro}
Among comparison theorems,
we mainly focus on the Bishop--Gromov volume comparison theorem and the Cheng maximal diameter theorem.
In the classical case,
the rigidity of the Bonnet--Myers theorem is known to be Cheng's maximal diameter rigidity (see Cheng \cite{cheng1975eigenvalue}).
These results are obtained for Riemannian manifolds with Ricci curvature bounded from below.
Subsequently,
these results were generalized to weighted Riemannian manifolds with weighted Ricci curvature bounded from below.
We first introduce geometric analysis on weighted Riemannian manifolds.
Let $(M,g)$ be an $n$-dimensional Riemannian manifold with $n\geq 2$,
$f\in C^\infty(M)$ and $v_g$ be the Riemannian volume measure.
The triple $(M,g,\e^{-f}v_g)$ is called a weighted Riemannian manifold.
For $N \in (-\infty,1] \cup (n,+\infty]$,
the $N$-weighted Ricci curvature is defined as follows:
\begin{align*}
    \Ric_{g,f}^N := \Ric_g + \Hess f - \frac{\d f \otimes \d f}{N-n},
\end{align*}
where the last term vanishes when $N = +\infty$.
$\Ric_{g,f}^N$ is also called $N$-Bakry--\'{E}mery Ricci curvature.
The case $N = +\infty$ was introduced by Lichnerowicz \cite{lichnerowicz1970varieses} and Bakry-\'{E}mery \cite{bakry1985diffusions},
and the case $N \in (n,+\infty)$ was introduced by Qian \cite{qian1997estimates}.
For $N \in (n,+\infty)$,
the comparison geometric properties of weighted Riemannian manifolds with $\Ric_{g,f}^N \geq (n-1)\kappa \,g$ are similar to those of Riemannian manifolds with $\Ric_g \geq (n-1)\kappa \,g$ and $\mathrm{dim}(M) \leq N$.
For $N_1 \in (-\infty,1)$ and $N_2 \in (n,+\infty)$,
we have 
\begin{align*}
    \Ric_{g,f}^{N_2} \leq \Ric_{g,f}^\infty \leq \Ric_{g,f}^{N_1} \leq \Ric_{g,f}^1.
\end{align*}
Therefore,
a lower bound of $\Ric_{g,f}^{1}$ is a weaker condition than the lower bound of $\Ric_{g,f}^{N_1}$.
In the past decade,
the case $N \in (-\infty,1]$ has been intensively studied.
Indeed,
Ohta \cite{ohta2016convexity} and Kolesnikov--Milman \cite{kolesnikov2018poincare} investigated the case $N \in (-\infty,0]$.
Furthermore,
Wylie \cite{wylie2017splitting} obtained the Cheeger--Gromoll splitting theorem for the case $N = 1$.
Later,
for $\varphi \in C^\infty(M)$,
Wylie--Yeroshkin \cite{wylie2016geometry} considered the following affine connection:
\begin{align}\label{eq:WY-conn}
    \nabla^{g,\varphi}_X Y := \nabla^g_X Y - \d \varphi (X) Y - \d \varphi (Y)X
\end{align}
where $\nabla^g$ is the Levi-Civita connection of $g$.
For an affine connection $D$,
the curvature tensor and the Ricci curvature are defined as follows:
\begin{align*}
    R^{D}(X,Y)Z := D_X D_Y Z - D_Y D_X Z - D_{[X,Y]}Z, \quad \Ric_g^{D}(Y,Z) := \sum_{i=1}^n g(R^{D}(E_i,Y)Z,E_i),
\end{align*}
where $\{E_i\}_{i=1}^n$ is a $g$-orthonormal frame.
We sometimes write $\Ric^D$ instead of $\Ric_g^D$ since \textcolor{black}{this trace does not depend on $g$.}
Wylie--Yeroshkin \cite{wylie2016geometry} showed 
\begin{align}\label{eq:WY-affine-ricci}
        \Ric_{g,f}^1 = \Ric^{\nabla^{g,\varphi}}
\end{align}
when $\varphi := (n-1)^{-1}f$.
We note that $\nabla^{g,\varphi}$ is said to be {projectively equivalent to} $\nabla^g$,
and projectively equivalent connections have long been studied.
Wylie--Yeroshkin \cite{wylie2016geometry} used a property obtained by Weyl that states that geodesics of $\nabla^g$ and $\nabla^{g,\varphi}$ coincide as unparameterized curves (see e.g., Eastwood \cite{eastwood2008notes}).
In \cite{wylie2016geometry},
by considering the geodesics of the affine connection $\nabla^{g,\varphi}$,
they obtained the Bishop--Gromov volume comparison theorem,
Bonnet--Myers theorem and the Cheng maximal diameter rigidity for weighted manifolds with 
\begin{align}\label{eq:WY-bound}
    \Ric_{g,f}^1 \geq (n-1)\kappa\, \e^{-\frac{4f}{n-1}}\,g.
\end{align}
We also refer to Lu--Minguzzi--Ohta \cite{lu2022comparison} and Kuwae--Sakurai \cite{kuwae2021rigidity} for these theorems.
For the case $N = 1$,
the corresponding sectional curvature was also studied in Wylie \cite{wylie2015sectional}, 
and further investigated by Kennard--Wylie \cite{kennard2017positive} and Kennard--Wylie--Yeroshkin \cite{kennard2019weighted}.
In this paper,
we study rigidity properties of comparison theorems in \cite{wylie2016geometry,kuwae2021rigidity} in the setting of centro-affine differential geometry.

We now move on to centro-affine differential geometry.
Following Klein's Erlangen program,
which classifies geometries by groups of transformations,
\c{T}i\c{t}eica employed the {centro-affine transformation}:
\begin{align*}
    x \mapsto A x \quad \text{for}\quad x \in \mathbb{R}^{n+1},\ A \in GL_{n+1}(\mathbb{R}),
\end{align*}
and studied the invariants under this transformation.
The works of \c{T}i\c{t}eica are considered to be among the origins of affine differential geometry (see Agnew--Bobe--Boskoff--Suceav\u{a} \cite{agnew2009gheorghe}).
In particular,
the surface studied by \c{T}i\c{t}eica is now recognized as a proper affine sphere,
which is a generalized notion of the usual sphere.
Affine differential geometry studies invariants under affine transformations: $x\mapsto A x + a$,
whereas centro-affine differential geometry studies invariants under centro-affine transformations.
The choice of transversal vector field is important in affine differential geometry.
\textcolor{black}{Blaschke introduced a transversal vector field,
which is called the affine normal or Blaschke normal for hypersurfaces.}
In centro-affine differential geometry,
we take the position vector field as the transversal vector field.
The Blaschke geometry and centro-affine differential geometry coincide on proper affine spheres with centers at the origin (see also Li-Li-Simon \cite{li2004centro}).

One of the main topics in affine differential geometry is the theory of affine spheres.
Affine spheres are divided into three classes,
which are called elliptic,
parabolic and hyperbolic.
\textcolor{black}{As observed in Loftin \cite[Proposition 3]{loftin2010survey}},
elliptic and hyperbolic affine spheres with centers at the origin can be locally given by the radial graph:
\begin{align}\label{eq:gigena-radial-graph}
    F_h(x) := \frac{1}{h(x)} (1,x) \in \mathbb{R}^{n+1},
\end{align}
for $x\in \Omega$,
where $\Omega$ is a domain in $\mathbb{R}^n$,
and $h$ is a solution of a Monge--Amp\`{e}re type equation.
\textcolor{black}{We note that Loftin \cite[Proposition 5]{loftin2010survey} also mentioned that the Blaschke connection on the proper affine sphere is projectively flat.
Given that the relation of projective equivalence in \eqref{eq:WY-conn} played an effective role in the geometric analysis by Wylie--Yeroshkin \cite{wylie2016geometry},
this suggests a natural intersection between the $1$-weighted Ricci curvature $\Ric_{g,f}^1$ and the affine differential geometry.}

Among the three classes,
elliptic,
parabolic and hyperbolic,
the only global examples of elliptic and parabolic affine spheres are the hyperquadrics,
namely ellipsoids and paraboloids (see Trudinger--Wang \cite{trudinger2008monge} and Loftin \cite{loftin2010survey}).
Indeed,
J\"{o}rgens \cite{jorgens1954uberdie} and Calabi \cite{calabi1958improper} obtained the parabolic case,
whereas Deicke \cite{deicke1953uberdie},
Calabi \cite{calabi1972complete} and Cheng--Yau \cite{cheng1986complete} obtained the elliptic case.
On the other hand,
in addition to hyperboloids,
there are plenty of examples of hyperbolic affine spheres.
In fact,
\c{T}i\c{t}eica already produced an example of a hyperbolic affine sphere (see Loftin \cite[Section 3]{loftin2010survey}).
This example was later generalized by Calabi \cite{calabi1972complete}.
Furthermore,
Calabi \cite{calabi1972complete} conjectured that,
for each proper convex cone,
there exists a hyperbolic affine sphere associated with it (see also \cite{loftin2011cheng}) and Cheng--Yau \cite{cheng1977regularity} solved this conjecture.
In \cite{cheng1977regularity},
the results by Cheng--Yau \cite{cheng1976regularity} concerning the Minkowski problem were used.
Indeed,
this is an intersection between the Minkowski problem and affine differential geometry.
Very recently,
Milman \cite{milman2023centro} discovered another intersection between them.

We move on to the theory developed by Milman \cite{milman2023centro} for the log-Minkowski problem.
The log-Minkowski problem is the case $p=0$ of the $L^p$-Minkowski problem.
The case $p=1$ corresponds to the classical Minkowski problem,
for which Cheng--Yau \cite{cheng1976regularity} proved regularity results of the convex body.
Lutwak \cite{lutwak1993brunnminkowski} introduced the $L^p$-Minkowski problem for the case $p > 1$.
For the case $p \leq 1$,
we refer to Chou--Wang \cite{chou2006minkowski}.
These Minkowski problems are related to the Brunn--Minkowski inequality.
Hilbert also proved the classical Brunn--Minkowski inequality using the first eigenvalue of an operator (see \textcolor{black}{Bonnesen--Fenchel \cite{bonnesen1987theory}} and Kolesnikov--Milman \cite{kolesnikov2022local} for the history).
Kolesnikov--Milman \cite{kolesnikov2022local} modified the operator,
and called it the {Hilbert--Brunn--Minkowski operator}.
Furthermore,
they investigated the relation between the eigenvalue of this operator and the $L^p$-Minkowski problem.
After that,
Milman \cite{milman2023centro} showed that the Hilbert--Brunn--Minkowski operator coincides with the centro-affine Laplacian associated with the convex body.
Using this interpretation,
the log-Minkowski problem was investigated in \cite{milman2023centro}.
This is the intersection between the Minkowski problem and centro-affine differential geometry mentioned above.
In particular,
the Bochner formula by Opozda \cite{opozda2015bochner},
which was developed in the theory of affine differential geometry,
was used in \cite{milman2023centro}.
The advantage of considering centro-affine differential geometry in convex geometry is that convex bodies have constant positive centro-affine Ricci curvature.

The aim of this paper is to conduct comparison geometric arguments in the framework of centro-affine differential geometry.
Our main results are summarized as follows:
\begin{itemize}
    \item[(i)]
    Our first main result concerns the Bishop--Gromov type volume comparison theorem for the radial graphs $F_h(\Omega)$.
    We equip the radial graphs with the centro-affine connection,
    and consider the Bishop--Gromov type volume comparison theorem for weighted manifolds with \eqref{eq:WY-bound}.
    We see that the equality in the volume comparison forces the radial graphs to be rotationally symmetric hyperquadrics (Theorem \ref{thm:rigidity-volume-comparison}).  
    These hyperquadrics include upper halves of ellipsoids and the upper sheets of the two-sheeted hyperboloids.
    \vspace{0.5\baselineskip} 
    \item[(ii)] 
    Our second main result concerns the Cheng type maximal diameter theorem on convex bodies.
    We equip the convex bodies with the weighted structure induced by the theory in Milman \cite{milman2023centro},
    and consider the Cheng type maximal diameter theorem for weighted manifolds with \eqref{eq:WY-bound} for $\kappa > 0$.
    The rigidity forces the convex body to be a rotationally symmetric ellipsoid (Theorem \ref{thm:cheng-milman}).  
\end{itemize}

The organization of this paper is as follows:
In Section \ref{sec:Preliminaries},
we introduce notation and review comparison geometry,
affine differential geometry,
and centro-affine differential geometry for convex bodies.
In Section \ref{sec:radial-graph},
we equip the radial graph of the form \eqref{eq:gigena-radial-graph} with a weighted structure,
whose Wylie--Yeroshkin type affine connection \eqref{eq:WY-conn} coincides with the centro-affine connection (Proposition \ref{prop:radial-graph}).
After that,
we state our first main result (Theorem \ref{thm:rigidity-volume-comparison}),
where we investigate rigidity phenomena of Bishop--Gromov type volume comparison in the radial graph setting.
Furthermore,
we see that the model spaces appearing in Theorem \ref{thm:rigidity-volume-comparison} actually satisfy the rigidity conditions (Proposition \ref{prop:quadrics}),
and we show that there is an example that satisfies the curvature bound \eqref{eq:WY-bound} but does not attain equality in the volume comparison (Proposition \ref{prop-radial-graph-defect}).
We also calculate the curvature of the Calabi example in our setting,
and show that it does not satisfy any curvature bound of the type \eqref{eq:WY-bound}  (Remark \ref{rem:Calabi-curvature-bound}).
In Section \ref{sec:Milman},
we equip convex bodies with a weighted structure.
In particular,
we observe that an affine connection introduced by Milman \cite{milman2023centro} coincides with the Wylie--Yeroshkin type affine connection \eqref{eq:WY-conn} (Proposition \ref{prop:conn-identity}).
This is another intersection between centro-affine differential geometry and $1$-weighted Ricci curvature.
We note that the $1$-weighted Ricci curvature associated with it coincides with the centro-affine Ricci curvature.
While this is different from the radial graph setting in Section \ref{sec:radial-graph},
we see that there is a relation between them (Proposition \ref{prop:milman-radial}).
After that,
we state our second main result (Theorem \ref{thm:cheng-milman}),
where we show the Cheng type maximal diameter theorem for convex bodies.
Furthermore,
we compute the diameter bounds of various ellipsoids.
We see that some ellipsoids provide examples which satisfy the curvature bound \eqref{eq:WY-bound} but do not attain the maximal diameter (Proposition \ref{prop:ellipsoid-milman}).
The results in Sections \ref{sec:radial-graph} and \ref{sec:Milman} provide several examples illustrating 
the intersection between centro-affine differential geometry and comparison geometry for $1$-weighted Ricci curvature.
\section{Preliminaries}\label{sec:Preliminaries}
In this section,
we review comparison geometry,
affine differential geometry and centro-affine differential geometry for convex bodies.
Throughout this paper,
we assume $n \geq 2$.
\subsection{Comparison geometry}\label{subsec:comparison}
Let $(M,g)$ be an $n$-dimensional complete Riemannian manifold and $f\in C^\infty(M)$.
We fix a point $p\in M$.
We denote the weighted measure by 
\begin{align*}
    \mathfrak{m}_{f,g} := \e^{-f}v_g.
\end{align*}
In order to state comparison theorems,
we prepare some notation following Lu--Minguzzi--Ohta \cite{lu2022comparison} and Kuwae--Sakurai \cite{kuwae2021rigidity}.
For $\kappa \in \mathbb{R}$,
we set 
\begin{align*}
    \mathrm{sn}_\kappa (s) := \begin{cases}
        \frac{1}{\sqrt{\kappa}} \sin(\sqrt{\kappa}\,s) &\text{ for } \kappa > 0,\\
        s &\text{ for } \kappa = 0,\\
        \frac{1}{\sqrt{-\kappa}} \sinh(\sqrt{-\kappa}\,s) &\text{ for } \kappa < 0,
    \end{cases}
    \qquad 
    \mathrm{ct}_\kappa(s) := \frac{\mathrm{sn}_\kappa'(s)}{\mathrm{sn}_\kappa (s)}.
\end{align*}
For a constant 
\begin{align*}
    C_\kappa := \inf\{ s > 0\ |\ \mathrm{sn}_\kappa (s) = 0\} = \begin{cases}
        \frac{\pi}{\sqrt{\kappa}} & \text{ for }\kappa > 0,\\
        +\infty & \text{ for } \kappa \leq 0,
    \end{cases}
\end{align*}
we set 
\begin{align*}
    \overline{\mathrm{sn}}_\kappa (s) := \begin{cases}
        \mathrm{sn}_\kappa (s) & \text{ for } 0 \leq s < C_\kappa,\\
        0 & \text{ for } s\geq C_\kappa.
    \end{cases}
\end{align*}
Let $g_{\mathbb{S}^{n-1}}$ be the standard Riemannian metric on $\mathbb{S}^{n-1}$ and 
\begin{align*}
    \omega_{n-1} := v_{g_{\mathbb{S}^{n-1}}} (\mathbb{S}^{n-1}).
\end{align*}
For $r\geq 0$,
we set 
\begin{align*}
    v(n,\kappa, r) := \omega_{n-1} \int_{0}^r \overline{\mathrm{sn}}_\kappa(\xi)^{n-1} \,\d \xi.
\end{align*}
We now move on to the {re-parametrization} by Wylie--Yeroshkin \cite{wylie2016geometry}.
For $\theta \in U_p M := \{ \theta \in T_p M \ |\ |\theta|_g = 1 \}$,
let $\gamma_{p,\theta}$ be a unit speed geodesic with $\gamma_{p,\theta}(0) = p$ and $\gamma_{p,\theta}'(0) = \theta$.
We set 
\begin{align*}
    \tau_{p}(\theta) := \sup \{ t > 0 \ |\ d_g(p,\gamma_{p,\theta}(t)) = t \} .
\end{align*}
Furthermore,
for $t\geq 0$ and $\theta \in U_p M$,
we define
\begin{align}\label{eq:def-s-tau}
     s_{f,p,\theta}(t) := \int_0^t \e^{-\frac{2f(\gamma_{p,\theta}(\xi))}{n-1}} \,\d \xi,\quad \tau_{f,p}(\theta) := s_{f,p,\theta} (\tau_{p}(\theta)).
\end{align}
We have the Bonnet--Myers and Cheng type rigidity properties as follows (see Wylie--Yeroshkin \cite[Theorem 2.2 and Theorem 4.16]{wylie2016geometry},
Lu--Minguzzi--Ohta \cite[Theorem 3.6]{lu2022comparison} and Kuwae--Sakurai \cite[Theorem 3.3 and Corollary 3.4]{kuwae2021rigidity}):
\begin{theorem}[\cite{wylie2016geometry,kuwae2021rigidity}]\label{thm:cheng}
Let $(M,g,\mathfrak{m}_{f,g})$ be an $n$-dimensional complete weighted Riemannian manifold.
For $\kappa > 0$,
we assume 
\begin{align*}
    \Ric_{g,f}^1  \geq (n-1) \kappa \,\e^{-\frac{4f}{n-1}}g.
\end{align*}
We set $\widehat{g} :=  \e^{-\frac{4f}{n-1}}g$.
Then we have 
\begin{align*}
    \mathrm{diam}(M,\widehat{g}) \leq \frac{\pi}{\sqrt{\kappa}}.
\end{align*}
We further assume that there exist $p,q\in M$ such that $f(p) = 0$ and 
\begin{align*}
    d_{\widehat{g}}(p,q) = \frac{\pi}{\sqrt{\kappa}}.
\end{align*}
Then $M$ is homeomorphic to $\mathbb{S}^{n}$,
$f$ is radial with respect to $p$,
and we have
\begin{align}\label{eq:metric:WY-thm-2-6}
    g = \d t^2 + \e^{\frac{2f(\gamma_{p,\theta}(t))}{n-1}} \mathrm{sn}_\kappa^2 \left( s_{f,p,\theta} (t)\right)g_{\mathbb{S}^{n-1}}
\end{align}
for $0 < t < d_g(p,q)$,
where $(t,\theta)$ is the geodesic polar coordinate around $p$.
\end{theorem}
We define the {re-parametrized distance} as follows: 
\begin{align}\label{eq:def-reparametrized-distance}
    s_{f,p}(x) := \inf_\gamma \int_0^{d_g(p,x)} \e^{-\frac{2f(\gamma(\xi))}{n-1}}\,\d \xi,
\end{align}
where $\gamma : [0,d_g(p,x)] \rightarrow M$ runs over all unit speed minimal geodesics from $p$ to $x$.
We define the associated ball as follows:
\begin{align*}
    B_{g,f,r}(p) := \{ x\in M\ |\ s_{f,p}(x) < r \}.
\end{align*}
We denote the ball by $B_{g,r}(p) := \{ x\in M \ |\ d_g(p,x) < r \}$.
For $t \in (0,\tau_p(\theta))$,
let $J_p(t,\theta)$ be the volume element of the $t$-level set of $d_g(p, \cdot)$ at $\gamma_{p,\theta}(t)$,
and let 
\begin{align}\label{eq:def-jacobian}
    J_{f,p}(t,\theta) := \e^{-f(\gamma_{p,\theta}(t))} J_p(t,\theta).
\end{align}
Also,
for $s \in (0, \tau_{f,p}(\theta))$,
we set 
\begin{align}\label{eq:def-jacobian-2}
    \widehat{J}_{f,p}(s, \theta) := J_{f,p}(t_{f,p,\theta}(s),\theta),
\end{align}
where $t_{f,p,\theta}$ is the inverse function of $s_{f,p,\theta}$.
Furthermore,
we set 
\begin{align}\label{eq:def-truncated-jacobian}
    \overline{J}_{f,p} (s,\theta) := \begin{cases}
        \widehat{J}_{f,p}(s,\theta) &\text{ for } s < \tau_{f,p}(\theta),\\
        0 &\text{ for } s \geq \tau_{f,p}(\theta),
    \end{cases} 
\end{align}
The Bishop--Gromov type volume comparison theorem is as follows (see Wylie--Yeroshkin \cite[Corollary 4.6]{wylie2016geometry},
Lu--Minguzzi--Ohta \cite[Theorem 3.11]{lu2022comparison} and Kuwae--Sakurai \cite[Proposition 4.4]{kuwae2021rigidity}):
\begin{theorem}[\cite{wylie2016geometry,lu2022comparison,kuwae2021rigidity}]\label{thm:BG}
Let $(M,g,\mathfrak{m}_{f,g})$ be an $n$-dimensional complete weighted Riemannian manifold.
For $\kappa \in\mathbb{R} $,
we assume 
\begin{align*}
    \Ric_{g,f}^1  \geq (n-1) \kappa \,\e^{-\frac{4f}{n-1}}g.
\end{align*}
For $p\in M$ and \textcolor{black}{$0 < r < R$},
we have 
\begin{align*}
    \frac{\mathfrak{m}_{\frac{n+1}{n-1}f, g} \left( B_{g,f,R}(p) \right)}{\mathfrak{m}_{\frac{n+1}{n-1}f, g} \left( B_{g,f,r}(p) \right)} \leq \frac{v(n,\kappa,R)}{v(n,\kappa,r)}.
\end{align*}
\end{theorem}
Although this subsection does not review the rigidity results for the Bishop--Gromov type volume comparison theorem,
we note that rigidity properties of Bishop--Gromov type volume comparison theorems are known (see \cite[Theorem 4.17 and 4.18]{wylie2016geometry} and \cite[Theorem 4.8]{kuwae2021rigidity}).
The metric of the form \eqref{eq:metric:WY-thm-2-6} for $\kappa \in \mathbb{R}$ appears in each rigidity phenomenon.
\subsection{Affine differential geometry}\label{subsec:affine-differential-geometry}
In this subsection,
we introduce basic notions in affine differential geometry.
Let $x : M\rightarrow \mathbb{R}^{n+1}$ be an immersion,
where $M$ is an $n$-dimensional smooth manifold.
Also,
we denote \textcolor{black}{the standard flat connection} on $\mathbb{R}^{n+1}$ by $\overline{D}$.
A smooth vector field $\xi$ along $x$ is called transverse if
\begin{align*}
    \xi_p \notin x_* (T_p M)
\end{align*}
for every $p\in M$.
For the transverse vector field $\xi$,
there exist 
a shape operator $S^\xi \in \Gamma(T^*M \otimes T M)$,
a transversal connection form $\tau^\xi \in \Omega^1(M)$,
a symmetric bilinear form $\mathfrak{g}^\xi$ on $T M$,
and a torsion-free affine connection $\nabla^\xi$ on $M$ such that they satisfy the following Gauss--Weingarten type equation:
\begin{align}\label{eq:affin-GW}
    \overline{D}_X (x_* Y) &= x_* (\nabla^\xi_X Y) + \mathfrak{g}^\xi (X,Y)\xi,\\
    \overline{D}_X \xi &= - x_* (S^\xi X) + \tau^{\xi} (X)\xi.\nonumber
\end{align}
\begin{remark}
If $\overline{\delta}$ denotes the standard Euclidean metric on $\mathbb{R}^{n+1}$,
then $\overline{D} = \nabla^{\overline{\delta}}$,
where $\nabla^{\overline{\delta}}$ is the Levi-Civita connection of $\overline{\delta}$.
Also,
there is another convention in \eqref{eq:affin-GW}.
Indeed,
$\overline{D}_X (x_* Y)$ means $(x^* \overline{D})_X (x_* Y)$,
where $x^* \overline{D}$ is the pullback connection of $\overline{D}$ by $x$.
\end{remark}
Below,
we consider locally strongly convex immersion $x$.
We follow the argument in Loftin \cite{loftin2010survey} and Benoist--Hulin \cite{benoist2013cubic}.
A transverse vector field $\xi$ is called the {affine normal} if the following three conditions hold:
\begin{itemize}
    \item[(a)] $\mathfrak{g}^\xi$ is positive  definite.
    \item[(b)] $\tau^\xi = 0$.
    \item[(c)] $\det_{1 \leq i,j \leq n}  \mathfrak{g}^\xi (E_i,E_j)   = \det \left( x_* E_1, \cdots, x_* E_n, \xi \right)^2$ for every frame $\{E_i\}_{i=1}^n$ of $T M$.
\end{itemize}
For the affine normal $\xi$,
the associated connection $\nabla^\xi$ and metric $\mathfrak{g}^\xi$ are called the {Blaschke connection} and {affine metric},
which we denote by $\nabla^B$ and $\mathfrak{g}^B$.
The {affine mean curvature} is defined by $H^\xi := \frac{1}{n}\mathrm{tr} S^\xi$.
Furthermore,
the hypersurface is called an {affine sphere} if $S^\xi = H \mathrm{Id}$ for some constant $H$.
In particular,
$M$ is called {elliptic} if $H > 0$,
{parabolic} if $H = 0$,
and {hyperbolic} if $H < 0$.
The point where all affine normal lines intersect is called the center.
An affine sphere without the center is said to be {improper}.
Otherwise,
it is called {proper}.
We have the following local representation using the radial graphs of the form \eqref{eq:gigena-radial-graph} (see Loftin \cite[Proposition 3]{loftin2010survey} and Loftin--McIntosh \cite[Subsection 2.1.4]{loftin2016cubic}):
\begin{theorem}[\cite{loftin2010survey}]\label{thm:classification}
We have the following assertions:
\begin{itemize}
    \item[(i)] A hyperbolic affine sphere with center $0$ and affine mean curvature $-1$ is locally given by the radial graph:
       $ F_h(\Omega) := \left\{ F_h(x) \ |\ x = (x_1, \cdots, x_n) \in \Omega  \right\}$,
    where $\Omega$ is a domain in $\mathbb{R}^n$ and $h \in \textcolor{black}{C^\infty(\Omega)}$ is a concave positive function satisfying the following Monge--Amp\`{e}re equation:
    \begin{align}\label{eq:thm:classification-1}
        \det_{1 \leq i,j \leq n}\left(  - \partial_i \partial_j h  \right) = h^{-(n+2)}.
    \end{align}
    Furthermore,
    the affine metric is written as follows:
    \begin{align*}
        \mathfrak{g}^B(\partial_i, \partial_j) = - \frac{\partial_i \partial_j h}{h}.
    \end{align*}
    \item[(ii)] An elliptic affine sphere with center $0$ and affine mean curvature $1$ is \textcolor{black}{locally} given by the radial graph:
    $F_h(\Omega) := \{ F_h(x)\ |\ x = (x_1, \cdots, x_n) \in \Omega\}$,
    where $\Omega$ is a domain in $\mathbb{R}^n$ and $h \in C^\infty(\Omega)$ is a convex positive function satisfying the following Monge--Amp\`{e}re equation:
    \begin{align}\label{eq:thm:classification-2}
        \det_{1 \leq i,j \leq n} \left( \partial_i \partial_j h \right)= h^{-(n+2)}.
    \end{align}
    Furthermore,
    the affine metric is written as follows:
    \begin{align*}
        \mathfrak{g}^B(\partial_i, \partial_j) =  \frac{\partial_i \partial_j h}{h}.
    \end{align*}
\end{itemize}
\end{theorem}
\begin{remark}\label{rem:affine-normal-radial-graph}
We explicitly calculate the affine normal.
For $\eps \in \{\pm{1}\}$,
we denote $\xi_{\eps, h} := -\eps F_h$.
The case $\eps = -1$ corresponds to Theorem \ref{thm:classification} (i) and the case $\eps  = 1$ corresponds to Theorem \ref{thm:classification} (ii).
We show that $\xi_{\eps, h}$ is the affine normal in each respective case. 
We check the three conditions (a) (b) (c) in the definition of the affine normal above.
It follows that
\begin{align}
    \overline{D}_{\partial_i} ((F_h)_* \partial_j) &=  (F_h)_* \left( -\frac{\partial_i h}{h}\partial_j - \frac{\partial_j h}{h}\partial_i \right) + \frac{\eps \partial_i \partial_j h }{h} \xi_{\eps,h},\label{eq:rem-GW-1}\\
    \overline{D}_{\partial_i} \xi_{\eps, h} &= - (F_h)_* (\eps \partial_i).\label{eq:rem-GW-2}
\end{align}
We also refer to the subsequent Lemma \ref{lem:radial-graph} for calculations of \eqref{eq:rem-GW-1}.
This implies the conditions (a) and (b).
Moreover,
we have 
\begin{align}\label{eq:rem:affine-normal-radial-graph-1}
    \det \left( (F_h)_* \partial_1, \cdots, (F_h)_* \partial _n, \xi_{\eps,h} \right)^2 = \det \left( (F_h)_* \partial_1, \cdots, (F_h)_* \partial _n, F_h \right)^2 = h^{-2(n+1)}.
\end{align}
On the other hand,
together with Theorem \ref{thm:classification},
we arrive at
\begin{align*}
    \det_{1 \leq i,j \leq n} \mathfrak{g}^{\xi_{\eps,h}} \left( \partial_i, \partial_j  \right) = h^{-n} \det (\eps \Hess_\delta h) = h^{-2(n+1)}.
\end{align*}
Combining this with \eqref{eq:rem:affine-normal-radial-graph-1},
we see that the condition (c) holds.
\end{remark}
\begin{example}
For $x  \in\mathbb{B}^n := \{ x \in \mathbb{R}^n \ |\ |x| < 1 \}$,
we set 
\begin{align*}
    h(x) := \sqrt{1 - |x|^2}.
\end{align*}
This solves the Monge--Amp\`{e}re equation \eqref{eq:thm:classification-1},
and the associated radial graph \eqref{eq:gigena-radial-graph} is the upper sheet of the two-sheeted hyperboloid (see Sasaki \cite[Example 2]{sasaki1980hyperbolic}).
\end{example}
\begin{example}
For $x \in \mathbb{R}^n$,
we set
\begin{align*}
    h(x) := \sqrt{1 + |x|^2}.
\end{align*}
This solves the Monge--Amp\`{e}re type equation \eqref{eq:thm:classification-2}.
We see that the associated radial graph $F_h$ is the upper hemisphere of the \textcolor{black}{unit sphere} (see also Loftin--McIntosh \cite[Section 2.1.4]{loftin2016cubic} and Loftin \cite[Proposition 3]{loftin2010survey}).
\end{example}
The projective flatness of the Blaschke connection is known as follows (see Loftin \cite[Proposition 5]{loftin2010survey}):
\begin{theorem}[\cite{loftin2010survey}]
The Blaschke connection on a proper affine sphere is projectively flat.
\end{theorem}
\begin{remark}
Here,
we explicitly write down the Blaschke connection $\nabla^B$ on the radial graph $F_h$.
We use the notation in Remark \ref{rem:affine-normal-radial-graph}.
Also,
let $\delta$ denote the standard flat metric on $\mathbb{R}^n$,
and we denote the standard flat connection on $\mathbb{R}^n$ by $\nabla^\delta$.
We note that $\nabla^\delta$ is the Levi-Civita connection of $\delta$.
It follows from \eqref{eq:rem-GW-1} that
\begin{align*}
    \left((F_h)^*(\nabla^B)\right)_{\partial_i} \partial_j = -\frac{\partial_i h}{h}\partial_j - \frac{\partial_j h}{h}\partial_i,
\end{align*}
where $(F_h)^*(\nabla^B)$ is the pull-back of $\nabla^B$ by $F_h$.
This implies 
\begin{align}
    \left((F_h)^*(\nabla^B)\right)_X Y = \textcolor{black}{\nabla^\delta}_X Y - X (\log h)Y - Y (\log h)X = \nabla^{\delta, \log h}_X Y,
\end{align}
where $\nabla^{\delta, \log h}$ is defined in \eqref{eq:WY-conn}.
Since $\nabla^\delta$ is the flat connection,
$(F_h)^*(\nabla^B)$ is projectively flat.
In particular,
$\nabla^B$ is projectively flat.
\end{remark}
As mentioned in Section \ref{sec:intro},
there are many examples of global hyperbolic affine spheres.
Indeed,
for a bounded convex domain $\Omega$ in $\mathbb{R}^n$,
Cheng--Yau \cite{cheng1977regularity} showed that there exists a unique positive concave solution $h \in C^\infty(\Omega) \cap C(\overline{\Omega})$ to the following Dirichlet problem of the Monge--Amp\`{e}re equation:
\begin{align}\label{eq:CY-boundary-problem}
        \det_{1 \leq i,j \leq n} \left( -\partial_i \partial_j h \right) = h^{-(n+2)},\quad h|_{\partial \Omega} = 0.
\end{align}
\textcolor{black}{Gigena \cite{gigena1981conjecture} showed} that the radial graph $F_h$ produces an example of a hyperbolic affine sphere asymptotic to the cone $\{ t(1,x) \ |\ x \in \Omega, t > 0\}$ (see also Loftin \cite[Section 7]{loftin2010survey} and Benoist--Hulin \cite[Theorem 2.5]{benoist2013cubic}).
\begin{example}
Calabi \cite{calabi1972complete} considered the following example:
\begin{align}\label{eq:Calabi-example}
    \left\{ (x_i)_{0 \leq i \leq n} \in \mathbb{R}^{n+1}\ \bigg|\  x_i > 0,\ \prod_{i=0}^{n} x_i = C \right\}.
\end{align}
The case $n=2$ was also considered by \c{T}i\c{t}eica (see \cite{loftin2010survey}).
This is a hyperbolic affine sphere.
Indeed,
we see that the Calabi example is realized as a radial graph of a Cheng--Yau type solution of \eqref{eq:CY-boundary-problem} on a simplex (see also Sasaki \cite[Example 3]{sasaki1980hyperbolic} and Loftin \cite{loftin2010survey}).
\end{example}
\begin{remark}\label{rem:Calabi-example-CY-2}
\textcolor{black}{For the Calabi example,
there is another viewpoint.}
For $\mathcal{O}_n := (0,+\infty)^n$ and $x \in \mathcal{O}_n$,
we set 
\begin{align*}
    h_{\mathcal{O}_n}(x) := \sqrt{n+1} \left(\prod_{i=1}^n x_i\right)^{\frac{1}{n+1}}.
\end{align*}
By direct calculations,
we see $h_{\mathcal{O}_n}$ satisfies the Monge--Amp\`{e}re equation \eqref{eq:thm:classification-1} (see e.g., Li--Simon--Zhao--Hu \cite[Example 3.1]{li2015global}). 
Also,
the radial graph $\frac{1}{h_{\mathcal{O}_n}(x)}(1,x)$ with $x \in \mathcal{O}_n$ coincides with the Calabi example \eqref{eq:Calabi-example} with suitable constant $C$.
\end{remark}
\subsection{Centro-affine differential geometry for convex bodies}\label{preliminary:milman}
We now move on to centro-affine differential geometry.
When the position vector field:
\begin{align*}
    \xi_x (p) := x(p)
\end{align*}
is a transversal vector field,
we call $\nabla^{\xi_x}$ the {centro-affine connection}.
\begin{remark}
\textcolor{black}{
On proper affine spheres with center at the origin,
the centro-affine connection and the Blaschke connection coincide since the affine normal is a constant multiple of the position vector as we observed in Remark \ref{rem:affine-normal-radial-graph}.}
\end{remark}
The sign convention of $\mathfrak{g}^\xi$ is that of Loftin \cite{loftin2010survey}.
In Milman \cite{milman2023centro},
the opposite sign is used for the centro-affine metric.
In this subsection,
we follow the sign convention in \cite{milman2023centro}.
In order to avoid confusion,
we denote
\begin{align*}
    g^\xi := - \mathfrak{g}^\xi.
\end{align*}
\textcolor{black}{For non-degenerate $g^\xi$},
the {$g^\xi$-dual connection} $(\nabla^\xi)^*$ is defined by the following relation:
\begin{align*}
    X (g^\xi (Y,Z)) = g^\xi(\nabla^\xi_X Y, Z) + g^\xi (Y, (\nabla^\xi)^*_X Z).
\end{align*}
We note that the centro-affine connection is invariant under the centro-affine transformation.
Indeed,
\textcolor{black}{we have the following assertion}:
\begin{proposition}\label{prop:centro-affine-transformation}
Let $x : M \rightarrow \mathbb{R}^{n+1}$ be an immersion,
where $M$ is an $n$-dimensional smooth manifold.
We assume $\xi_x := x$ is a transversal vector field.
For $A \in GL(\mathbb{R}^{n+1})$,
we set $\widetilde{x} := A \circ x$.
Then we have 
\begin{align*}
    \nabla^{\xi_x} = \nabla^{\xi_{\widetilde{x}}},\quad g^{\xi_x} = g^{\xi_{\widetilde{x}}}.
\end{align*}
\end{proposition}

Below,
we introduce terminologies following \cite{milman2023centro}.
Let $E$ be a $(n+1)$-dimensional vector space over $\mathbb{R}$.
In this paper,
we identify $E$ with $\mathbb{R}^{n+1}$.
Let $E^*$ denote its dual,
and we also often identify $E^*$ with $E$.
We denote unit spheres in $E$ and $E^*$ by $\mathbb{S}$ and $\mathbb{S}^*$,
respectively.
Let $\mathcal{K}$ be the set of convex bodies containing the origin in their interiors,
and let $\mathcal{K}_+^\infty$ be the set of convex bodies in $\mathcal{K}$ with $C^\infty$-smooth boundary and strictly positive curvature.
We note that,
for $K \in \mathcal{K}_+^\infty$,
$\partial K$ is strongly convex.
For $K \in \mathcal{K}$,
the {support function} is defined as follows:
\begin{align*}
    h_K(x^*) := \max_{x \in K} \langle x^*,x \rangle
\end{align*}
for $x^* \in E^*$,
where $\langle , \rangle$ denotes the natural pairing between $E$ and $E^*$.
For $e \in \mathbb{S}^*$,
we set $e^{\perp} := \{ x \in E \ |\ \langle e,x \rangle = 0\}$,
and $P_{e^{\perp}} : E \rightarrow E$ denotes the projection to the subspace $e^{\perp}$.
We note that we may regard $h_K$ as a function on $\mathbb{S}^*$ by restriction.
For $K \in \mathcal{K}_+^\infty$,
we denote the position vector by 
\begin{align*}
    x_{\partial K}^K(p) := p
\end{align*}
for $p\in \partial K$.
Also,
we denote the associated centro-affine connection and the bilinear form by $\nabla^{\partial K}_K$ and $g^{\partial K}_K$.
Let $K^* := \{ x^* \in E^*\ |\ h_K(x^*) \leq 1\}$.
For $M \in \mathcal{M}_K := \{ \mathbb{S}, \mathbb{S}^*, \partial K, \partial K^* \}$,
Milman \cite[Section 4]{milman2023centro} constructed diffeomorphisms $T_K^{M_i \rightarrow M_j} : M_i \rightarrow M_j$ for $M_i, M_j \in \mathcal{M}_K$ and set $x_K^M := T_{K}^{M\rightarrow \partial K}$.
We note that $T_{K}^{\partial K \rightarrow \mathbb{S}^*}$ coincides with the Gauss map.
We define 
\begin{align*}
    g_K^M := (x_K^M)^* g_K^{\partial K},\quad \nabla_K^M := (x_K^M)^* \nabla_K^{\partial K}.
\end{align*}
\begin{remark}
For $M_1, M_2 \in \mathcal{M}_K$,
we have 
\begin{align*}
    g_K^{M_1} = (T_K^{M_1 \rightarrow M_2})^* g_K^{M_2},\quad \nabla_K^{M_1} = (T_K^{M_1 \rightarrow M_2})^* \nabla_K^{M_2}.
\end{align*} 
\end{remark}
Convex bodies have constant positive centro-affine Ricci curvature as follows (see \cite[Subsection 3.10]{milman2023centro}):
\begin{proposition}[\cite{milman2023centro}]\label{prop:milman-constant}
Let $K \in \mathcal{K}_+^{\infty}$.
For $M \in \mathcal{M}_K$,
we have 
\begin{align}
    \Ric^{\nabla_K^M} = \Ric^{(\nabla_K^M)^*} = (n-1) g_K^M,
\end{align}
where $(\nabla_K^M)^*$ is the $g_K^M$-dual connection of $\nabla_K^M$.
\end{proposition}
Below,
we denote the Hessian of $h_K$ with respect to $\overline{D}$ by $\overline{D}^2 h_K$,
and the restriction of $\overline{D}^2 h_K$ to $T \mathbb{S}^*$ by $D^2 h_K$.
The metric and connection are explicitly written down as follows (see \cite[Proposition 4.2]{milman2023centro}):
\begin{proposition}[\cite{milman2023centro}]\label{prop:milman-prop4-2}
Let $K\in \mathcal{K}_+^\infty$.
We have 
\begin{align*}
    g_K^{\mathbb{S}^*} &= \frac{D^2 h_K}{h_K} = \frac{\Hess_{g_{\mathbb{S}^*}}h_K}{h_K} + g_{\mathbb{S}^*},\\
    (\nabla_K^{\mathbb{S}^*})^*_X Y &= \nabla^{g_{\mathbb{S}^*}}_X Y - (X \log h_K) Y - (Y \log h_K) X,
\end{align*}
where $g_{\mathbb{S}^*}$ is the \textcolor{black}{standard Riemannian metric} on $\mathbb{S}^*$,
and $(\nabla_K^{\mathbb{S}^*})^*$ is the $g_K^{\mathbb{S}^*}$-dual connection of $\nabla_K^{\mathbb{S}^*}$.
\end{proposition}
\section{Radial graphs}\label{sec:radial-graph}
In this section,
we equip the radial graph $F_h(\Omega)$ with a weighted structure and discuss rigidity properties of Bishop--Gromov type volume comparison theorems.
\subsection{Weighted structure on radial graphs}
In Subsection \ref{subsec:affine-differential-geometry},
we calculated the Blaschke connection on the radial graphs.
Below,
we calculate the centro-affine connection.
The centro-affine connection on the radial graph is related to the Wylie--Yeroshkin type affine connection as follows:
\begin{proposition}\label{prop:radial-graph}
Let $\Omega$ be a domain in $\mathbb{R}^n$ and $0< h \in C^\infty(\Omega)$.
We set 
\begin{align*}
    f_h := (n-1)\log h,\quad \overline{f}_h := f_h \circ F_h^{-1}, \quad \overline{g}_h := (F_h)_* \delta.
\end{align*}
Also,
we denote the position vector field on $F_h(\Omega)$ by $\xi_h$.
Then $\xi_h$ is a transversal vector field and
\begin{align}\label{eq:thm-radial-graph}
    \nabla^{\xi_h} = \nabla^{\overline{g}_h, \frac{\overline{f}_h}{n-1}},\quad g^{\xi_h} = \frac{1}{n-1} \Ric_{\overline{g}_h, \overline{f}_h}^1.
\end{align}
\end{proposition}
\begin{remark}
In Proposition \ref{prop:radial-graph},
$F_h$ is an embedding.
Indeed,
since 
\begin{align}\label{eq:rem:radial-graph}
    (F_h)_* X = \frac{1}{h}(0,X) - \frac{\d h(X)}{h^2} (1,x),
\end{align}
we see that $F_h$ is an immersion.
Also,
for $(y_0,y)\in F_h(\Omega)$ with $y_0\in \mathbb{R}$ and $y \in \mathbb{R}^n$,
the map $(y_0,y) \mapsto y_0^{-1} y$ gives the inverse function of $F_h$.
\end{remark}
We first prove the following lemma:
\begin{lemma}\label{lem:radial-graph}
We assume the same assumptions as in Proposition \ref{prop:radial-graph}.
We have
\begin{align*}
    \overline{D}_X ((F_h)_* Y) = (F_h)_* \left( \nabla^{\delta}_X Y - \frac{\d f_h (X)}{n-1} Y - \frac{\d f_h(Y)}{n-1}X  \right) - \frac{\Hess_{\delta} h(X,Y)}{h}F_h.
\end{align*}
\end{lemma}
\begin{proof}
For $x = (x_1, \cdots, x_n) \in \Omega$,
we set $\partial_i := \frac{\partial}{\partial x_i}$.
For the immersion $\iota : \Omega \rightarrow \mathbb{R}^{n+1}$ with $\iota(x) := (1,x)$,
let $E_i := \iota_* (\partial_i)$.
For brevity of notation,
we denote $F_h := F$,
$F_i := (F_h)_* (\partial_i)$ and $h_i := \partial_i h$.
Using this,
we have
\begin{align*}
    \overline{D}_{\partial_i} F_j &= \partial_i \left( \frac{1}{h} \right)E_j - \partial_i \left( \frac{h_j}{h} \right)F_h - \frac{h_j}{h}\overline{D}_{\partial_i} F \\
    &= - \frac{h_i}{h^2} E_j - \left( \frac{h_{ij}}{h} - \frac{h_i h_j}{h^2}\right)F - \frac{h_j}{h}F_i\\
    &= - \frac{h_i}{h^2} \left( h F_j + h_j F \right) - \left( \frac{h_{ij}}{h} - \frac{h_i h_j}{h^2}\right)F - \frac{h_j}{h}F_i\\
    &= - \frac{h_i}{h} F_j - \frac{h_j}{h}F_i - \frac{h_{ij}}{h}F, 
\end{align*}
where we used
\begin{align}
    F_j = \frac{1}{h}E_j - \frac{h_j}{h}F
\end{align}
in the third equality.
Here,
we note that the symbol $\overline{D}_{\partial_i} F_j$ denotes the pullback connection $(F^* \overline{D})_{\partial_i}F_j$.
This yields the desired assertion.
\end{proof}
We are now in a position to prove Proposition \ref{prop:radial-graph}.
\begin{proof}[Proof of Proposition \ref{prop:radial-graph}]
We first observe that $\xi_h$ is transversal.
Indeed,
if not,
there exist $x \in \Omega$ and $X \in T \Omega$ such that $(F_h)_* X = F_h(x)$.
Together with \eqref{eq:rem:radial-graph},
we see that comparison of the first coordinate gives $\d h(X) = -h$,
and comparison of the remaining coordinates gives $X = 0$.
This implies $h(x) = 0$,
which is a contradiction.

It follows from Lemma \ref{lem:radial-graph} that 
\begin{align}
    \nabla^{\xi_h}_{(F_h)_* X} ((F_h)_* Y) &= (F_h)_* \left( \nabla^\delta_X Y - \frac{\d f_h(X)}{n-1}Y - \frac{\d f_h(Y)}{n-1}X \right),\label{eq:thm-radial-graph-1}\\
    (F_h^* g^{\xi_h})(X,Y) &= \frac{\Hess_\delta h(X,Y)}{h}.\label{eq:thm-radial-graph-2}
\end{align}
Since $F_h$ is an isometry from $(\Omega,\delta)$ to $(F_h(\Omega), \overline{g}_h)$,
\eqref{eq:thm-radial-graph-1} implies the \textcolor{black}{first identity} in \eqref{eq:thm-radial-graph}.
For the second identity,
we have 
\begin{align}\label{eq:Ricci-Hess-relation}
    (F_h)^* (\Ric_{\overline{g}_h, \overline{f}_h}^1) = \Ric_{\delta,f_h}^1 = \Hess_\delta f_h + \frac{\d f_h \otimes \d f_h}{n-1} = (n-1) \frac{\Hess_\delta h}{h} =(n-1) F_h^* g^{\xi_h},
\end{align}
where we used \eqref{eq:thm-radial-graph-2} in the last equality.
We conclude the proof.
\end{proof}
\subsection{Rigidity of volume comparison theorem on radial graphs}
We first formulate the volume comparison theorem on non-complete Riemannian manifolds.
This is needed since the domain $\Omega$ appearing in the radial graph setting is not always complete.
We keep the notation in Subsection \ref{subsec:comparison} with the following modification for the non-complete case.
Let $(M,g)$ be an $n$-dimensional Riemannian manifold.
For $p\in M$ and $\theta \in U_p M$,
let $\gamma_{p,\theta} : [0,\rho_p (\theta)) \rightarrow M$ be the maximal unit-speed geodesic with $\gamma_{p,\theta}(0) = p$ and $\gamma'_{p,\theta}(0) = \theta$,
where  
\begin{align*}
    \rho_p (\theta) := \sup \{ t > 0\ |\  \gamma_{p,\theta} \text{ is defined on } [0,t)\}.
\end{align*}
We note that,
for a complete Riemannian manifold,
we have $\rho_p (\theta) = +\infty$.
We define 
\begin{align*}
    \tau_p(\theta) := \sup\{ t \in [0,\rho_p(\theta))\ |\ d_g(p,\gamma_{p,\theta}(t)) = t \}.
\end{align*}
For $t \in [0,\tau_p(\theta))$,
we define $s_{f,p,\theta}(t)$ in the same manner as in \eqref{eq:def-s-tau}.
Also,
we set 
\begin{align*}
    \tau_{f,p}(\theta) := \lim_{t \nearrow \tau_p(\theta)} s_{f,p,\theta}(t).
\end{align*}
Also,
in the same manner as in \eqref{eq:def-jacobian}, 
\eqref{eq:def-jacobian-2} and \eqref{eq:def-truncated-jacobian},
we define  $J_{f,p}, \widehat{J}_{f,p}$ and $\overline{J}_{f,p}$.
Since $M$ is not assumed to be complete,
a minimizing geodesic from $p$ to an arbitrary point may not exist.
Hence,
the definition of the re-parametrized distance in \eqref{eq:def-reparametrized-distance} does not work.
For $r \geq 0$,
instead of $B_{g,f,r}(p)$,
we define
\begin{align*}
    C_{g,f,r}(p) := \{ \gamma_{p,\theta}(t)\ |\ \theta \in U_p M,\ 0 \leq t < \tau_p (\theta),\ s_{f,p,\theta}(t) < r\}.
\end{align*}
\begin{remark}
For a complete Riemannian manifold,
we have  $\rho_p (\theta) = +\infty$ and the notation $\tau_p, \tau_{f,p}, s_{f,p,\theta}$ defined above coincide with those defined in Subsection \ref{subsec:comparison}.
Furthermore,
for $r \geq 0$,
we have 
\begin{align*}
    B_{g,f,r}(p) \backslash C_{g,f,r}(p) \subset \mathrm{Cut}(p),
\end{align*}
where $\mathrm{Cut}(p)$ denotes the cut locus of $p$.
\end{remark}
We have the following volume comparison:
\begin{corollary}\label{cor:C-BG}
Let $(M,g,\mathfrak{m}_{f,g})$ be an $n$-dimensional Riemannian manifold.
For $\kappa \in\mathbb{R}$,
we assume 
\begin{align*}
    \Ric_{g,f}^1  \geq (n-1) \kappa \,\e^{-\frac{4f}{n-1}}g.
\end{align*}
For $p\in M$ and $0 < r < R$,
we have 
\begin{align*}
    \frac{\mathfrak{m}_{\frac{n+1}{n-1}f, g} \left( C_{g,f,R}(p) \right)}{\mathfrak{m}_{\frac{n+1}{n-1}f, g} \left( C_{g,f,r}(p) \right)} \leq \frac{v(n,\kappa,R)}{v(n,\kappa,r)}.
\end{align*}
\end{corollary}
\begin{proof}
The proof of Theorem \ref{thm:BG} relies on the Riccati inequality (see Wylie--Yeroshkin \cite[Lemma 4.1]{wylie2016geometry}, 
Lu--Minguzzi--Ohta \cite[Proposition 3.5]{lu2022comparison},
Kuwae--Sakurai \cite[Lemma 2.1]{kuwae2021rigidity}).
In our non-complete setting,
the same Riccati inequality holds along minimizing geodesics.
Also,
we have 
\begin{align*}
     \textcolor{black}{\mathfrak{m}_{\frac{n+1}{n-1}f, g}\left( C_{g,f,r}(p) \right)  = \int_{U_pM} \int_0^r \overline{J}_{f,p}(s,\theta) \,\d s \,\d \theta}.
\end{align*}
Hence,
applying the same argument in proving Theorem \ref{thm:BG} leads us to the desired conclusion. 
\end{proof}
We now move on to the radial graph setting on a convex domain $\Omega$ in $\mathbb{R}^n$ with $0\in \Omega$.
We note that we have $\rho_0(\theta) = \tau_0(\theta)$ in this setting.
We set 
\begin{align*}
    \tau_\kappa := \begin{cases}
        \frac{\pi}{2\sqrt{\kappa}} &\text{ for }\kappa > 0,\\
        +\infty &\text{ for } \kappa \leq 0,
    \end{cases}
    \qquad 
    \Omega(\kappa) := \begin{cases}
        \mathbb{R}^n &\text{ for }\kappa \geq 0,\\
        B_{\delta,\frac{1}{\sqrt{-\kappa}}}(0)& \text{ for }\kappa < 0.
    \end{cases}
\end{align*}
We note that $B_{\delta,\frac{1}{\sqrt{-\kappa}}}(0) = \{ x \in \mathbb{R}^n \ |\ |x| < \frac{1}{\sqrt{-\kappa}} \}$.
We have the following volume rigidity:
\begin{theorem}\label{thm:rigidity-volume-comparison}
Let $\Omega$ be a convex domain in $\mathbb{R}^n$ with $0\in \Omega$,
and $0 < h \in C^\infty(\Omega)$.
We set 
\begin{align*}
    f_h := (n-1)\log h,\quad \overline{f}_h := f_h \circ F_h^{-1},\quad \overline{g}_h := (F_h)_* \delta.
\end{align*}
For $\kappa \in \mathbb{R}$,
we assume 
\begin{align}\label{eq:thm:rigidity-volume-comparison-0}
    \Ric_{\overline{g}_h, \overline{f}_h}^1 \geq (n-1) \kappa \,\e^{-\frac{4\overline{f}_h}{n-1}}\overline{g}_h.
\end{align}
For every $0 < r < R < \tau_\kappa$,
we assume 
\begin{align}\label{eq:thm:rigidity-volume-comparison-1}
    \frac{\mathfrak{m}_{\frac{n+1}{n-1}\overline{f}_h, \overline{g}_h} \left( C_{\overline{g}_h, \overline{f}_h, R}\left( F_h(0) \right) \right)}{\mathfrak{m}_{\frac{n+1}{n-1}\overline{f}_h, \overline{g}_h} \left( C_{\overline{g}_h, \overline{f}_h, r}\left( F_h(0) \right) \right)} = \frac{v(n,\kappa,R)}{v(n,\kappa,r)}.
\end{align}
Then $F_h(\Omega)$ is a \textcolor{black}{rotationally symmetric} hyperquadric.
In particular,
we have 
\begin{align*}
    h(x) = h(0)\sqrt{1 +  \widetilde{\kappa} |x|^2},
    \qquad
    \Omega = \Omega(\widetilde{\kappa}),
\end{align*}
where we set $\widetilde{\kappa} := h(0)^{-4}\kappa$.
\end{theorem}
We first prove this theorem under the normalization $h(0) = 1$.
In order to do so,
we prove Lemmas \ref{lem:sn} to \ref{lem:domain}.
In Lemma \ref{lem:sn},
we show an identity of the Jacobian.
Using it,
we determine $h$ locally in Lemmas \ref{lem:quadrics} and \ref{lem:quadrics-further}.
Then we determine $h$ globally on $\Omega$ in Lemma \ref{lem:globalization}.
Lastly,
by using this,
we determine the shape of $\Omega$ in Lemma \ref{lem:domain}.
All these arguments are conducted under the assumption of the normalization $h(0)=1$.
After that,
we move on to the general case that does not assume $h(0)=1$.
Indeed,
we prove Theorem \ref{thm:rigidity-volume-comparison} by applying lemmas to $\widetilde{h} := h(0)^{-1}h$.

Throughout the following lemmas (Lemma \ref{lem:sn} to \ref{lem:domain}),
we assume the same assumptions as in Theorem \ref{thm:rigidity-volume-comparison},
and we further assume $h(0) = 1$.
The first lemma is as follows (see also Kuwae--Sakurai \cite[Theorem 4.8]{kuwae2021rigidity}):
\begin{lemma}\label{lem:sn}
For $\theta \in \mathbb{S}^{n-1}$ and $0 < s < \min\{\tau_{f_h,0}(\theta), \tau_\kappa\}$,
we have 
\begin{align*}
     \widehat{J}_{f_h,0} (s,\theta) = \mathrm{sn}_\kappa^{n-1}(s).
\end{align*}
\end{lemma}
\begin{proof}
It is enough to work on $(\Omega, \delta)$ since $(F_h(\Omega),\overline{g}_h)$ is isometric to $(\Omega,\delta)$.
We take $r > 0$ such that $s < r < \min\{ \tau_{f_h,0}(\theta), \tau_\kappa\}$.
It follows from Corollary \ref{cor:C-BG},
\eqref{eq:thm:rigidity-volume-comparison-1} and the assumption $h(0)=1$ that 
\begin{align}\label{eq:lem:quadrics-1}
    \mathfrak{m}_{\frac{n+1}{n-1}f_h,\delta}\left( C_{\delta,f_h,r}(0) \right) = v(n,\kappa,r).
\end{align}
For this,
we also refer to the argument in Kuwae--Sakurai \cite[Lemma 4.1]{kuwae2021rigidity}.
This implies 
\begin{align}\label{eq:thm:rigidity-volume-comparison-2}
    \int_{\mathbb{S}^{n-1}} \int_0^r \overline{J}_{f_h,0}(\xi,\eta)\,\d \xi\d \eta = \int_{\mathbb{S}^{n-1}} \int_{0}^r \overline{\mathrm{sn}}_\kappa^{n-1}(\xi)\,\d \xi\d \eta.
\end{align}
On the other hand,
by the same argument as in Kuwae--Sakurai \cite[Lemma 4.1]{kuwae2021rigidity},
we have 
\begin{align}\label{eq:thm:rigidity-volume-comparison-2-5}
    \overline{J}_{f_h,0} (\xi,\eta) 
    \leq 
    \overline{\mathrm{sn}}_\kappa^{n-1}(\xi)
\end{align}
for $(\xi,\eta) \in (0,r) \times \mathbb{S}^{n-1}$.
Together with \eqref{eq:thm:rigidity-volume-comparison-2},
we have 
\begin{align}\label{eq:lem:sn-1}
    \overline{J}_{f_h,0} (\xi,\eta) = \overline{\mathrm{sn}}_\kappa^{n-1}(\xi) 
\end{align}
for almost every $(\xi,\eta) \in (0,r) \times \mathbb{S}^{n-1}$.
Since $r < \tau_\kappa$,
we see $\overline{\mathrm{sn}}_\kappa^{n-1}(\xi)  > 0$.
Hence,
we have $\xi < \tau_{f_h,0}(\eta)$ almost everywhere.
This implies $\overline{J}_{f_h,0} (\xi,\eta) = \widehat{J}_{f_h,0} (\xi,\eta)$ almost everywhere.
It follows from the continuity of $\widehat{J}_{f_h,0} (\xi,\eta)$ that we obtain the desired identity.
\end{proof}
We have the following assertion:
\begin{lemma}\label{lem:quadrics}
For every $0 < r < \tau_\kappa$ and $x \in C_{\delta,f_h,r}(0)$,
we have 
\begin{align*}
    h(x)^2 = \left(1 + \left\langle \nabla h(0),x \right\rangle\right)^2 + \kappa |x|^2.
\end{align*}
\end{lemma}
\begin{proof}
For $\theta \in \mathbb{S}^{n-1}$,
$0 < s < \tau_{f_h,0}(\theta)$ and $t := t_{f_h,0,\theta}(s)$,
we have 
\begin{align}\label{eq:lem:quadrics-eq-1}
    \widehat{J}_{f_h,0}(s,\theta) = \e^{-f_h(t\theta)} J_0(t,\theta) = \frac{1}{h(t\theta)^{n-1}}\, t^{n-1}.
\end{align}
Hence, 
it follows from Lemma \ref{lem:sn} that
\begin{align}\label{eq:thm:rigidity-volume-comparison-3}
    \left( \frac{t}{h(t \theta)} \right)^{n-1} =\mathrm{sn}_\kappa^{n-1}(s)
\end{align}
for $0 < s < \min\left\{ \tau_\kappa, \tau_{f_h,0}(\theta) \right\}$.
Using \eqref{eq:thm:rigidity-volume-comparison-3} and $\mathrm{ct}'_\kappa(s) = - \mathrm{sn}_\kappa(s)^{-2}$,
we have 
\begin{align*}
    \frac{\d }{\d t} \mathrm{ct}_\kappa (s_{f_h, 0, \theta}(t)) 
    = - \frac{1}{\mathrm{sn}_\kappa^2 (s_{f_h,0,\theta}(t)) h(t\theta)^2}= -\frac{1}{t^2}.
\end{align*}
This implies 
\begin{align*}
    \mathrm{ct}_\kappa(s_{f_h, 0, \theta}(t)) = \frac{1}{t} + C(\theta),
\end{align*}
where $C(\theta)$ is a constant depending on $\theta$.
Finally,
again using \eqref{eq:thm:rigidity-volume-comparison-3} and $\mathrm{ct}_\kappa(s)^2 + \kappa = \mathrm{sn}_\kappa(s)^{-2}$,
we derive
\begin{align}\label{eq:thm:rigidity-volume-comparison-4}
    h(t\theta)^2 = t^2 (\mathrm{ct}_\kappa (s_{f_h,0,\theta}(t))^2 + \kappa) = 1 + 2 C(\theta) t + (C(\theta)^2 + \kappa)t^2.
\end{align}
Since $h \in C^\infty(\Omega)$ with $h(0) = 1$,
for sufficiently small $t$,
we have 
\begin{align*}
    h(t\theta) = 1 + t \langle \nabla h(0),\theta \rangle + O(t^2),
\end{align*}
which implies 
\begin{align*}
    h(t\theta)^2 = 1 + 2 t \langle \nabla h(0),\theta \rangle + O(t^2).
\end{align*}
Comparing this with \eqref{eq:thm:rigidity-volume-comparison-4},
we have $C(\theta) = \langle \nabla h(0),\theta \rangle$.
Therefore,
for $\theta \in \mathbb{S}^{n-1}$,
$0 <s < \min\{\tau_\kappa, \tau_{f_h,0}(\theta)\}$ and $t := t_{f_h,0,\theta}(s)$,
we have 
\begin{align}\label{eq:lem:quadrics-eq-2}
    h(t\theta)^2 = (1 + t \langle \nabla h(0),\theta \rangle)^2 + \kappa t^2 . 
\end{align}

Now we fix $x \in C_{\delta,f_h,r}(0)$ with $x \neq 0$,
and we write $x = t\theta$ with $t > 0$ and $\theta \in \mathbb{S}^{n-1}$.
Since $x \in C_{\delta,f_h,r}(0)$,
we have $t < \tau_0(\theta)$ and $s_{f_h, 0,\theta}(t) < r$.
Hence,
we have 
\begin{align*}
    s_{f_h, 0,\theta}(t) < \min\{\tau_\kappa, \tau_{f_h,0}(\theta)\}.
\end{align*} 
Therefore,
combining this with \eqref{eq:lem:quadrics-eq-2},
we conclude the proof.
\end{proof}
Furthermore,
we have the following assertion:
\begin{lemma}\label{lem:quadrics-further}
We have $\nabla h(0) = 0$ when $\kappa \neq 0$.
\end{lemma}
\begin{proof}
We show this by contradiction.
We assume $\nabla h(0) \neq 0$.
We set $Q(x):= h(x)^2$ and fix $0 < r < \tau_\kappa$.
By Lemma \ref{lem:quadrics},
we see 
\begin{align*}
    \nabla Q(x) = 2 \left( 1 + \langle \nabla h(0),x \rangle \right) \nabla h(0) + 2\kappa x.
\end{align*}
for $x \in C_{\delta,f_h,r}(0)$.
There exists $\eps > 0$ such that $B_{\delta,\eps}(0) \subset C_{\delta, f_h, r}(0)$,
where $B_{\delta,\eps}(0) = \{x \in \mathbb{R}^n\ |\ |x| < \eps\}$.
We set $e := |\nabla h(0)|^{-1} \nabla h(0)$ and $x_t := t e$ for $ |t| < \eps$.
Then for $\theta \in \mathbb{S}^{n-1}$ with $\langle \theta, e \rangle = 0$,
we obtain 
\begin{align}\label{eq:thm:rigidity-volume-comparison-9}
    \d Q_{x_t}(\theta) = 0.
\end{align}
Moreover,
we have
\begin{align*}
    \Hess_\delta Q(\theta,\theta) = 2 \langle \nabla h(0),\theta \rangle^2 + 2\kappa |\theta|^2 = 2 \kappa.
\end{align*}
Also,
it follows that
\begin{align*}
    \Hess_\delta Q (\theta,\theta) = 2 h \Hess_{\delta} h(\theta,\theta) + 2 \d h(\theta)^2 = 2 h \Hess_{\delta} h(\theta,\theta),
\end{align*}
where we used \eqref{eq:thm:rigidity-volume-comparison-9} in the second equality.
Combining these,
we obtain 
\begin{align*}
    h(x_t) \left( \Hess_\delta h \right)_{x_t}(\theta,\theta) = \kappa.
\end{align*}
Therefore,
it follows that
\begin{align}\label{eq:thm:rigidity-volume-comparison-5}
    \left( \Hess_\delta h \right)_{x_t}(\theta,\theta) - \kappa h(x_t)^{-3} \delta(\theta, \theta) &= \kappa h(x_t)^{-3} \left( h(x_t)^2 - 1 \right)  \\
    &= \kappa h(x_t)^{-3} \left\{ 2 |\nabla h(0)| t + (|\nabla h(0)|^2 + \kappa )t^2 \right\},\nonumber
\end{align}
where we used Lemma \ref{lem:quadrics} in the last equality.
If $\kappa > 0$,
we choose $t < 0$ sufficiently close to $0$,
and if $\kappa < 0$,
we choose $t > 0$ sufficiently close to $0$,
then the right-hand side is negative in both cases.
On the other hand,
together with \eqref{eq:Ricci-Hess-relation},
the curvature condition \eqref{eq:thm:rigidity-volume-comparison-0} gives
\begin{align}
    \Hess_\delta h \geq \kappa h^{-3}\delta.
\end{align}
This implies \eqref{eq:thm:rigidity-volume-comparison-5} is non-negative.
This is a contradiction.
Therefore,
we see $\nabla h(0) = 0$.
We conclude the proof.
\end{proof}
We determine $h$ globally on $\Omega$ as follows:
\begin{lemma}\label{lem:globalization}
The assertion in Lemma \ref{lem:quadrics} holds for arbitrary $x \in \Omega$,
i.e.,
we have 
\begin{align}\label{eq:lem:globalization-1}
    h(x)^2 = \begin{cases}
        1 + \kappa |x|^2 &\text{ when } \kappa \neq 0,\\
        \left(1 + \langle \nabla h(0),x \rangle\right)^2 &\text{ when } \kappa =0.
    \end{cases}
\end{align}
\end{lemma}
\begin{proof}
We fix $x \in \Omega$ with $x \neq 0$.
We write $x = T\theta \in \Omega$ with $T > 0$ and $\theta \in \mathbb{S}^{n-1}$.
First,
we consider the case $\kappa > 0$.
We have 
\begin{align}\label{eq:thm:rigidity-volume-comparison-7}
    s_{f_h,0,\theta}(T) < \tau_\kappa.
\end{align}
We show this by contradiction.
We assume \eqref{eq:thm:rigidity-volume-comparison-7} does not hold.
Then,
there exists $T_* \in (0,T]$ such that $s_{f_h,0,\theta}(T_*) = \tau_\kappa$ and $s_{f_h,0,\theta}(t) < \tau_\kappa$ for $t \in (0,T_*)$.
For each $t \in (0,T_*)$,
we fix $r_t \in (s_{f_h,0,\theta}(t),\tau_\kappa)$.
Then we see $t \theta \in C_{\delta,f_h,r_t}(0)$.
Therefore,
Lemmas \ref{lem:quadrics} and \ref{lem:quadrics-further} give
\begin{align*}
    h(t\theta)^2 = 1 + \kappa t^2.
\end{align*}
Therefore,
for $t \in (0,T_*)$,
we have
\begin{align*}
    s_{f_h,0,\theta}(t) = \int_0^t \frac{1}{1 + \kappa \xi^2}\,\d \xi = \frac{1}{\sqrt{\kappa}}\mathrm{arctan}\left( \sqrt{\kappa}\,t \right).
\end{align*}
Letting $t \nearrow T_*$,
we obtain
\begin{align*}
    s_{f_h,0,\theta}(T_*) = \frac{1}{\sqrt{\kappa}}\mathrm{arctan}\left( \sqrt{\kappa}\,T_* \right) < \tau_\kappa,
\end{align*}
which contradicts the definition of $T_*$.
Therefore,
we have \eqref{eq:thm:rigidity-volume-comparison-7}.
Then for $r \in (s_{f_h,0,\theta}(T),\tau_\kappa)$,
we have $x \in C_{\delta,f_h,r}(0)$.
By Lemmas \ref{lem:quadrics} and \ref{lem:quadrics-further},
\eqref{eq:lem:globalization-1} holds on $C_{\delta,f_h,r}(0)$.
Therefore,
we arrive at the desired assertion.

Next,
we consider the case $\kappa \leq 0$.
Since $\tau_\kappa = +\infty$,
we have \eqref{eq:thm:rigidity-volume-comparison-7}.
Then the desired assertion follows by the same argument above.
\end{proof}
We determine the shape of $\Omega$ as follows:
\begin{lemma}\label{lem:domain}
We also have $\nabla h(0) = 0$ when $\kappa = 0$.
Furthermore,
we have $\Omega = \Omega(\kappa)$ for $\kappa \in \mathbb{R}$.
\end{lemma}
\begin{proof}
We take $0 < r < \tau_\kappa$ arbitrarily.
By the same argument in deriving \eqref{eq:lem:sn-1},
we have 
\begin{align*}
    \overline{J}_{f_h,0}(s,\theta) = \mathrm{sn}_\kappa^{n-1}(s)
\end{align*}
for almost every $(s,\theta) \in (0,r)\times \mathbb{S}^{n-1}$.
For $s > \tau_{f_h,0}(\theta)$,
the left-hand side is equal to $0$,
whereas the right-hand side is not.
This implies $\tau_{f_h,0}(\theta) \geq r$ for almost every $\theta$.
Since we take $0 < r < \tau_\kappa$ arbitrarily,
we arrive at
\begin{align}\label{eq:lem:domain-eq-1}
    \tau_{f_h,0}(\theta) \geq \tau_\kappa
\end{align}
for almost every $\theta$.

First,
we consider the case $\kappa > 0$.
We note that 
\begin{align*}
    \tau_0(\theta) = \sup\{ t > 0 \ | \ t\theta \in \Omega \}
\end{align*}
for $\theta \in \mathbb{S}^{n-1}$.
Therefore,
for $\theta \in \mathbb{S}^{n-1}$ and $0 \leq \xi < \tau_0(\theta)$,
we have $\xi \theta \in \Omega$.
Then Lemma \ref{lem:globalization} implies 
\begin{align*}
    h(\xi \theta)^2 = 1 + \kappa \xi^2.
\end{align*}
It follows that
\begin{align*}
    \tau_{f_h,0}(\theta) = \lim_{t \nearrow \tau_0(\theta)} \int_0^t h(\xi \theta)^{-2} \,\d \xi = \lim_{t \nearrow \tau_0(\theta)} \int_0^t \frac{1}{1 + \kappa \xi^2}\,\d \xi 
    \leq \frac{\pi}{2\sqrt{\kappa}} = \tau_\kappa.
\end{align*}
Together with \eqref{eq:lem:domain-eq-1},
we have $\tau_0(\theta) = +\infty$ for almost every $\theta$.

Next,
we consider the case $\kappa < 0$.
By \eqref{eq:lem:domain-eq-1} and $\tau_\kappa = +\infty$,
we have $\tau_{f_h,0}(\theta) = +\infty$ for almost every $\theta$.
On the other hand,
it follows from Lemma \ref{lem:globalization} that 
\begin{align*}
    \tau_{f_h,0}(\theta) = \lim_{t \nearrow \tau_0(\theta)} \int_0^t \frac{1}{1 + \kappa \xi^2} \,\d \xi = \lim_{t \nearrow \tau_0(\theta)} \left[ \frac{1}{\sqrt{-\kappa}} \mathrm{arctanh}\left( \sqrt{-\kappa}\,\xi \right)\right]_0^t.
\end{align*}
Hence,
we have $\tau_0(\theta) = \frac{1}{\sqrt{-\kappa}}$ for almost every $\theta$.

Lastly,
we consider the case $\kappa = 0$.
We show $\nabla h(0) = 0$ by contradiction.
We assume $\nabla h(0) \neq 0$.
We set $\Theta := \{ \theta \in \mathbb{S}^{n-1} \ |\ \langle \nabla h(0),\theta \rangle > 0\}$.
For $\theta \in \Theta$,
it follows from Lemma \ref{lem:globalization} that 
\begin{align}\label{eq:lem:domain-eq-4}
    \tau_{f_h,0}(\theta) = \lim_{t \nearrow \tau_0(\theta)} \int_0^t h(\xi \theta)^{-2} \,\d \xi = \lim_{t \nearrow \tau_0(\theta)} \int_0^t \frac{1}{\left( 1 + \xi \langle \nabla h(0), \theta\rangle \right)^2}\,\d \xi \leq \frac{1}{\langle \nabla h(0), \theta \rangle} < +\infty.
\end{align}
Since $\Theta$ has positive measure,
\eqref{eq:lem:domain-eq-4} contradicts \eqref{eq:lem:domain-eq-1}.
Hence,
we have $\nabla h(0) = 0$.
This implies $h \equiv 1$.
Then we have $\tau_0(\theta) = \tau_{f_h,0}(\theta)$ for $\theta \in \mathbb{S}^{n-1}$.
Therefore,
it follows from \eqref{eq:lem:domain-eq-1} that $\tau_0(\theta) = +\infty$ for almost every $\theta$.

Combining these preceding arguments,
we obtain 
\begin{align*}
    \tau_0(\theta) = \begin{cases}
        +\infty &\text{ when }\kappa \geq 0,\\
        \frac{1}{\sqrt{-\kappa}} &\text{ when }\kappa < 0.
    \end{cases}
\end{align*}
for almost every $\theta$.
Since $\Omega$ is convex and contains the origin,
we see $\Omega = \Omega(\kappa)$.
We conclude the proof.
\end{proof}
We are now in a position to give a proof of the theorem.
\begin{proof}[Proof of Theorem \ref{thm:rigidity-volume-comparison}]
We set 
\begin{align*}
    c := h(0),\quad \widetilde{h} := c^{-1}h,\quad \widetilde{f}_h := (n-1) \log \widetilde{h},\quad \widetilde{\kappa} := c^{-4}{\kappa}.
\end{align*}
We see 
\begin{align*}
    \Ric_{\delta, \widetilde{f}_h}^1 = \Ric_{\delta, f_h}^1 \geq (n-1) \kappa \,\e^{-\frac{4 f_h}{n-1}}\delta = (n-1) \widetilde{\kappa}\,\e^{-\frac{4\widetilde{f}_h}{n-1}}\delta.
\end{align*}
We also have 
\begin{align*}
    s_{\widetilde{f}_h, 0, \theta}(t) = c^2  s_{f_h,0,\theta}(t),\quad C_{\delta, \widetilde{f}_h, r}(0) = C_{\delta, f_h, c^{-2}r}(0), \quad \mathfrak{m}_{\frac{n+1}{n-1}\widetilde{f}_h, \delta} = c^{n+1} \mathfrak{m}_{\frac{n+1}{n-1}f_h,\delta}.
\end{align*}
For $0 < r < R < \tau_{\widetilde{\kappa}} = c^2 \tau_\kappa$,
this implies
\begin{align}\label{eq:thm:vol-rig-1}
    \frac{\mathfrak{m}_{\frac{n+1}{n-1}\widetilde{f}_h, \delta} \left( C_{\delta, \widetilde{f}_h, R}(0) \right)}{\mathfrak{m}_{\frac{n+1}{n-1}\widetilde{f}_h, \delta} \left( C_{\delta, \widetilde{f}_h, r}(0) \right)} = \frac{\mathfrak{m}_{\frac{n+1}{n-1}f_h,\delta} \left( C_{\delta, f_h, c^{-2}R}(0) \right)}{\mathfrak{m}_{\frac{n+1}{n-1}f_h,\delta} \left( C_{\delta, f_h, c^{-2}r}(0) \right)}.
\end{align}
Since $\mathrm{sn}_{\widetilde{\kappa}}\left( r \right) = c^2 \mathrm{sn}_{\kappa}\left( c^{-2}r \right)$,
we have $v(n,\widetilde{\kappa},r) = c^{2n} v(n,\kappa,c^{-2}r)$.
This implies 
\begin{align*}
    \frac{v(n,\widetilde{\kappa},R)}{v(n,\widetilde{\kappa},r)} = \frac{v(n,\kappa,c^{-2}R)}{v(n,\kappa,c^{-2}r)}.
\end{align*}
Combining this with \eqref{eq:thm:vol-rig-1} and the assumption \eqref{eq:thm:rigidity-volume-comparison-1},
we have 
\begin{align*}
    \frac{\mathfrak{m}_{\frac{n+1}{n-1}\widetilde{f}_h, \delta} \left( C_{\delta, \widetilde{f}_h, R}(0) \right)}{\mathfrak{m}_{\frac{n+1}{n-1}\widetilde{f}_h, \delta} \left( C_{\delta, \widetilde{f}_h, r}(0) \right)} = \frac{v(n,\widetilde{\kappa},R)}{v(n,\widetilde{\kappa},r)}.
\end{align*}
Hence,
Lemma \ref{lem:domain} applied to $\widetilde{h}$ gives $\Omega = \Omega(\widetilde{\kappa})$ and $\widetilde{h}(x) = \sqrt{1 + \widetilde{\kappa}|x|^2}$.
\textcolor{black}{Since $(\Omega, \delta)$ is isometric to $(F_h(\Omega), \overline{g}_h)$,}
we conclude the proof.
\end{proof}
\begin{remark}
We see that $F_h(\Omega(\widetilde{\kappa}))$ is a hyperquadric.
We set 
\begin{align*}
    M_{\kappa,c} := \left\{ (y_0,y) \in \mathbb{R} \times \mathbb{R}^n\ \bigg| \ y_0 > 0,\quad c^2 y_0^2 + \frac{\kappa}{c^2}|y|^2 = 1 \right\}.
\end{align*}
When $\kappa > 0$,
this is the upper half of a rotationally symmetric ellipsoid,
and when $\kappa < 0$,
this is the upper sheet of a rotationally symmetric two-sheeted hyperboloid.
Then we have $F_h(\Omega(\widetilde{\kappa})) = M_{\kappa,c}$.
Indeed,
for $(y_0,y) = F_h(x)$,
we have 
\begin{align*}
    c^2 y_0^2 + \frac{\kappa}{c^2} |y|^2 = \frac{c^2 + c^{-2}\kappa |x|^2}{h(x)^2} = 1.
\end{align*}
This implies $F_h(\Omega(\widetilde{\kappa})) \subset M_{\kappa,c}$.
On the other hand,
for $(y_0,y) \in M_{\kappa,c}$,
we set $x := y_0^{-1}y$.
Then we have $x \in \Omega(\widetilde{\kappa})$ and $F_h(x) = (y_0,y)$.
This implies $M_{\kappa,c} \subset F_h(\Omega(\widetilde{\kappa}))$.
\end{remark}
\subsection{Examples}\label{subsec:examples}
In this subsection,
we consider hyperquadrics that clarify the role of two conditions in Theorem \ref{thm:rigidity-volume-comparison}.
Although we showed in Theorem \ref{thm:rigidity-volume-comparison} that the curvature bound and equality in the volume comparison imply that the radial graph is a hyperquadric,
we observe the converse in the following proposition.
Indeed,
the hyperquadrics in Theorem \ref{thm:rigidity-volume-comparison} provide the equality cases for the volume comparison under the Wylie--Yeroshkin type curvature bound in \eqref{eq:WY-bound} as follows:
\begin{proposition}\label{prop:quadrics}
Let $c > 0$ and $\kappa \in \mathbb{R}$.
For $x \in \Omega(c^{-4}\kappa)$,
we set 
\begin{align*}
    h(x) := c\sqrt{1 + \frac{\kappa}{c^4}|x|^2}.
\end{align*}
Then \eqref{eq:thm:rigidity-volume-comparison-0} and \eqref{eq:thm:rigidity-volume-comparison-1} hold.
\end{proposition}
\begin{proof}
It is enough to consider only $(\Omega(c^{-4}\kappa),\delta)$.
We use the same notation as in Theorem \ref{thm:rigidity-volume-comparison}.
First,
we only consider the case $c = 1$. 
By direct calculations,
we have 
\begin{align*}
    \Hess_\delta h = \frac{\kappa}{h}\delta - \frac{\kappa^2}{h^3}x \otimes x.
\end{align*}
For $\theta \in \mathbb{S}^{n-1}$,
it follows that
\begin{align}\label{eq:prop:quadrics-4}
    \frac{1}{n-1} \Ric_{\delta, f_h}^1 (\theta,\theta) &= \frac{\Hess_\delta h(\theta,\theta)}{h}
     = \frac{\kappa}{h^2} - \frac{\kappa^2 \langle x , \theta\rangle^2}{h^4} 
    = \frac{\kappa + \kappa^2 (|x|^2 - \langle x, \theta\rangle^2)}{h^4}
    \geq \frac{\kappa}{h^4},
\end{align}
where we used \eqref{eq:Ricci-Hess-relation} and $h(x)^2 = 1 + \kappa |x|^2$.
This implies \eqref{eq:thm:rigidity-volume-comparison-0}.

We now move on to the proof of \eqref{eq:thm:rigidity-volume-comparison-1}.
For $x \in \Omega(\kappa)$ with $x \neq 0$,
we write $x = t \theta$ with $t > 0$ and $\theta \in \mathbb{S}^{n-1}$.
We have 
\begin{align}\label{eq:prop:quadrics-5}
    s_{f_h,0,\theta}(t) = \int_0^t \frac{1}{1 + \kappa \xi^2}\,\d \xi = \begin{cases}
        \frac{1}{\sqrt{\kappa}} \mathrm{arctan}\left( \sqrt{\kappa}\,t \right) &\text{ for } \kappa > 0,\\
        t &\text{ for } \kappa = 0,\\
        \frac{1}{\sqrt{-\kappa}} \mathrm{arctanh}\left( \sqrt{-\kappa}\,t \right)&\text{ for } \kappa < 0 .
    \end{cases}
\end{align}
We next show 
\begin{align}\label{eq:prop:quadrics-9}
    \widehat{J}_{f_h,0}(s,\theta) = \mathrm{sn}_\kappa^{n-1}(s)
\end{align}
for $ 0< s < \tau_{f_h,0}(\theta)$.
We set $ t := t_{f_h,0,\theta}(s)$.
Since we have the identity \eqref{eq:lem:quadrics-eq-1},
it is enough to show 
\begin{align}\label{eq:thm:rigidity-volume-comparison-00}
    \frac{t}{h(t\theta)} = \mathrm{sn}_\kappa(s).
\end{align}
We divide the proof into three cases depending on the sign of $\kappa$.
If $\kappa > 0$,
we have $\tan(\sqrt{\kappa}\,s) = \sqrt{\kappa}\,t$,
which implies 
\begin{align*}
    1 + \kappa t^2 = 1 + \tan (\sqrt{\kappa}\,s)^2 = \frac{1}{\cos (\sqrt{\kappa}\,s)^2}.
\end{align*}
Hence,
it follows that
\begin{align*}
    \frac{t}{h(t\theta)} = \frac{t}{\sqrt{1 + \kappa \,t^2}} = \frac{1}{\sqrt{\kappa}}\sin(\sqrt{\kappa}\,s).
\end{align*}
If $\kappa = 0$,
since $h \equiv 1$,
we have $s = t$. 
If $\kappa < 0$,
denoting $\lambda := \sqrt{-\kappa}$,
we have $\lambda \,t = \tanh(\lambda\,s)$.
Hence,
we see
\begin{align*}
    1 - \lambda^2 t^2 = 1 - \tanh (\lambda\,s)^2 = \frac{1}{\cosh(\lambda\,s)^2}.
\end{align*}
It follows that
\begin{align}\label{eq:thm:rigidity-volume-comparison-03}
    \frac{t}{h(t\theta)} = \frac{t}{\sqrt{1 - \lambda^2 t^2}} = \frac{1}{\lambda} \sinh(\lambda\,s).
\end{align}
Therefore,
we have \eqref{eq:thm:rigidity-volume-comparison-00},
which yields \eqref{eq:prop:quadrics-9}.
Since it follows from \eqref{eq:prop:quadrics-5} that we have $\tau_{f_h,0}(\theta) = \tau_\kappa$,
for every $\theta \in \mathbb{S}^{n-1}$,
the identity \eqref{eq:prop:quadrics-9} holds for all $0 < s < \tau_\kappa$.
Consequently,
for $0 < r < \tau_\kappa$,
we obtain
\begin{align}\label{eq:thm:rigidity-volume-comparison-02}
    \mathfrak{m}_{\frac{n+1}{n-1}f_h,\delta} \left( C_{\delta, f_h, r}(0) \right) = \int_{\mathbb{S}^{n-1}} \int_0^r \mathrm{sn}_\kappa^{n-1}(s)\,\d s \,\d \theta = \omega_{n-1} \int_0^r \mathrm{sn}_\kappa^{n-1}(s) \,\d s = v(n,\kappa,r).
\end{align}
Hence,
we arrive at the desired assertion for the case $c = 1$.
The case $c\neq 1$ follows by applying the same scaling argument as in the proof of Theorem \ref{thm:rigidity-volume-comparison}.
We conclude the proof.
\end{proof}
We now observe an example that satisfies the curvature bound but does not attain equality in the volume comparison theorem.
Indeed,
we have the following assertion:
\begin{proposition}\label{prop-radial-graph-defect}
For $a > 1$,
we set 
\begin{align*}
    A_a := \mathrm{diag}(a, a^{-1}, 1, \cdots, 1),\quad \Omega_a := \{x \in \mathbb{R}^n\ |\ \langle A_a x, x \rangle < 1\},\quad h_a(x) := \sqrt{1- \langle A_a x, x \rangle }.
\end{align*}
Also,
let 
\begin{align*}
    f_a := (n-1) \log h_a,\quad \overline{f}_a := f_a \circ F_{h_a}^{-1},\quad \overline{g}_a := (F_{h_a})_* \delta.
\end{align*}
Denoting $\kappa_a := -a$,
we have 
\begin{align}\label{eq:prop-radial-graph-defect-1}
    \Ric_{\overline{g}_a, \overline{f}_a}^1 \geq (n-1)\kappa_a \e^{-\frac{4\overline{f}_a}{n-1}}\overline{g}_a.
\end{align}
Furthermore,
for $0 < r < R$,
we have 
\begin{align}\label{eq:prop-radial-graph-defect-2}
    \frac{\mathfrak{m}_{\frac{n+1}{n-1}\overline{f}_a, \overline{g}_a} \left( C_{\overline{g}_a, \overline{f}_a, R} \left( F_{h_a}(0) \right) \right)}{\mathfrak{m}_{\frac{n+1}{n-1}\overline{f}_a, \overline{g}_a} \left( C_{\overline{g}_a, \overline{f}_a, r} \left( F_{h_a}(0) \right) \right)} < \frac{v(n,\kappa_a,R)}{v(n,\kappa_a,r)}.
\end{align}
\end{proposition}
\begin{proof}
It is enough to consider only $(\Omega_a, \delta)$.
For every $\theta \in \mathbb{S}^{n-1}$,
the same calculation as in \eqref{eq:prop:quadrics-4} yields
\begin{align}\label{eq:prop-radial-graph-defect-3}
    \frac{1}{n-1} \Ric_{\delta,f_a}^1  (\theta,\theta)= \frac{\Hess_\delta h_a(\theta,\theta)}{h_a} = - \frac{h_a^2\langle A_a \theta,\theta \rangle + \langle A_a x, \theta \rangle^2}{h_a^4}.
\end{align}
From the Cauchy--Schwarz inequality for $\langle A_a \cdot,\cdot\rangle$,
it follows that
\begin{align*}
    h_a^2\langle A_a \theta, \theta \rangle + \langle A_a x, \theta \rangle^2 \leq h_a^2 \langle A_a \theta,\theta \rangle + \langle A_a x,x \rangle \langle A_a \theta,\theta\rangle = \langle A_a \theta,\theta \rangle \leq a.
\end{align*}
Together with \eqref{eq:prop-radial-graph-defect-3},
we arrive at \eqref{eq:prop-radial-graph-defect-1}.

Next,
we show \eqref{eq:prop-radial-graph-defect-2}.
We set $\lambda(\theta) := \langle A_a \theta, \theta\rangle$.
By the same argument as in Proposition \ref{prop:quadrics},
we have 
\begin{align}\label{eq:prop-radial-graph-defect-5}
    s_{f_a,0,\theta}(t) 
    = \frac{1}{\sqrt{\lambda(\theta)}} \mathrm{arctanh}\left( \sqrt{\lambda(\theta)} \,t\right),\quad 
    \frac{t}{h_a (t\theta)} 
    = \mathrm{sn}_{-\lambda(\theta)}(s).
\end{align}
Together with the identity \eqref{eq:lem:quadrics-eq-1},
we see 
\begin{align}\label{eq:prop-radial-graph-defect-7}
    \widehat{J}_{f_a,0}(s,\theta) = \mathrm{sn}_{-\lambda(\theta)}^{n-1}(s)
\end{align}
for $0 < s < \tau_{f_a,0}(\theta)$.
Since \eqref{eq:prop-radial-graph-defect-5} implies $\tau_{f_a,0}(\theta) = +\infty$,
the identity \eqref{eq:prop-radial-graph-defect-7} holds for all $s > 0$.
Therefore,
for $r > 0$,
we obtain
\begin{align}\label{eq:prop-radial-graph-defect-6}
    \mathfrak{m}_{\frac{n+1}{n-1}f_a,\delta} \left( C_{\delta, f_a, r}(0) \right) = \int_{\mathbb{S}^{n-1}} \int_0^r \mathrm{sn}_{-\lambda(\theta)}^{n-1}(s)\,\d s \,\d \theta.
\end{align}
Since we have $\lambda(\theta) < a$ for almost every $\theta$,
we arrive at \eqref{eq:prop-radial-graph-defect-2}.
\end{proof}
\begin{remark}
The function $h_a$ is a solution of \eqref{eq:CY-boundary-problem}.
Indeed,
we have 
\begin{align*}
    -\Hess_\delta h_a = \frac{A_a}{h_a} + \frac{A_a x\otimes A_a x}{h_a^3} > 0.
\end{align*}
Since $h_a > 0$ on $\Omega_a$ and $h_a = 0$ on $\partial \Omega_a$,
it remains to prove the Monge--Amp\`{e}re equation.
Using $\det A_a = 1$ and the matrix determinant lemma,
it follows that
\begin{align*}
    \det(-\Hess_\delta h_a) = h_a^{-n} \det(A_a) \det \left( I + \frac{x \otimes A_a x}{h_a^2}\right) = h_a^{-n}\left( 1 + \frac{\langle A_a x,x \rangle}{h_a^2} \right) = h_a^{-(n+2)}.
\end{align*}
\end{remark}
\begin{remark}\label{rem:Calabi-curvature-bound}
We consider the Calabi example using the notation in Remark \ref{rem:Calabi-example-CY-2}.
For the Calabi example,
no Wylie--Yeroshkin type curvature bound \eqref{eq:thm:rigidity-volume-comparison-0} exists.
Indeed,
there is no constant $\kappa \in \mathbb{R}$ such that 
\begin{align}\label{eq:prop:Calabi-1}
    \Ric_{\overline{g}_{\mathcal{O}_n}, \overline{f}_{\mathcal{O}_n}}^1 \geq (n-1) \kappa\, \e^{-\frac{4\overline{f}_{\mathcal{O}_n}}{n-1}}\overline{g}_{\mathcal{O}_n},
\end{align}
where we set 
\begin{align*}
    f_{\mathcal{O}_n} := (n-1) \log h_{\mathcal{O}_n},\quad \overline{f}_{\mathcal{O}_n} := f_{\mathcal{O}_n} \circ F_{\mathcal{O}_n}^{-1}, \quad \overline{g}_{\mathcal{O}_n} := (F_{\mathcal{O}_n})_* \delta.
\end{align*}
Below,
we give a proof by contradiction.
We assume that there exists $\kappa \in \mathbb{R}$ satisfying \eqref{eq:prop:Calabi-1}.
It is enough to work on $(\mathcal{O}_n, \delta)$.
Together with \eqref{eq:Ricci-Hess-relation},
it follows that \eqref{eq:prop:Calabi-1} is equivalent to 
\begin{align*}
    \Hess_\delta h_{\mathcal{O}_n} \geq \kappa h_{\mathcal{O}_n}^{-3} \delta.
\end{align*}
Let $x(t) := (t,1, \cdots, 1)$ for $t > 0$ and $\alpha := (n+1)^{-1}$.
Since 
\begin{align*}
    h_{\mathcal{O}_n}(x(t)) = \sqrt{n+1}\,t^\alpha, \quad \partial_{11} h_{\mathcal{O}_n}(x(t)) = \alpha (\alpha - 1)\sqrt{n+1}\,t^{\alpha-2},
\end{align*}
we obtain 
\begin{align*}
    h_{\mathcal{O}_n}(x(t))^3 \partial_{11} h_{\mathcal{O}_n}\left( x(t) \right) = -n t^{-\frac{2(n-1)}{n+1}}.
\end{align*}
The right-hand side goes to $-\infty$ when $t \searrow 0$.
On the other hand,
\eqref{eq:prop:Calabi-1} implies 
\begin{align*}
     h_{\mathcal{O}_n}(x(t))^3 \partial_{11} h_{\mathcal{O}_n}\left( x(t) \right)  \geq \kappa,
\end{align*}
which is a contradiction.
\end{remark}
\section{Convex bodies}\label{sec:Milman}
In this section,
we study the Cheng maximal diameter theorem in the setting of Milman \cite{milman2023centro}.
\subsection{Weighted structure on convex bodies}\label{subsec:Milman-WY}
In this subsection,
we introduce a weighted structure on convex bodies.
We use the notation introduced in Subsection \ref{preliminary:milman}.
Let $K \in \mathcal{K}_+^\infty$,
and let $\mathfrak{n} : \partial K \rightarrow \mathbb{S}^*$ be the Gauss map.
We introduce notation associated with the standard Riemannian metric $g_{\mathbb{S}^*}$ and $\mathfrak{n}$ as follows:
\begin{align*}
    g_{\partial K} := \mathfrak{n}^* g_{\mathbb{S}^*},\quad h_{\partial K} := h_K \circ \mathfrak{n},\quad f_K := (n-1)\log h_K,\quad f_{\partial K} := (n-1) \log h_{\partial K}.
\end{align*}
Furthermore,
we set
\begin{align*}
    \widehat{g}_K^{\mathbb{S}^*} := \e^{-\frac{4f_K}{n-1}}g_{\mathbb{S}^*},\quad \widehat{g}_{\partial K} := \e^{-\frac{4f_{\partial K}}{n-1}}g_{\partial K}. 
\end{align*}
We have the following assertion:
\begin{proposition}\label{prop:conn-identity}
For $K \in \mathcal{K}_+^\infty$,
we have 
\begin{align}\label{eq:prop:conn-identity-1}
    (\nabla_K^{\partial K})^* = \nabla^{g_{\partial K},\frac{f_{\partial K}}{n-1}},
\end{align}
where $(\nabla_K^{\partial K})^*$ is the $g_K^{\partial K}$-dual connection of $\nabla_K^{\partial K}$.
Furthermore,
we have
\begin{align}\label{eq:prop:conn-identity-2}
    \Ric_{g_{\partial K},f_{\partial K}}^1 = \Ric^{(\nabla_K^{\partial K})^*} = \Ric^{\nabla_K^{\partial K}} = (n-1) g_K^{\partial K}.
\end{align}
\end{proposition}
\begin{proof}
We first prove the identities on $\mathbb{S}^*$,
then pull them back by $\mathfrak{n}$.
For brevity,
we write $D := \nabla_K^{\mathbb{S}^*}$,
and let $D^*$ be its $g_{K}^{\mathbb{S}^*}$-dual connection.
By Proposition \ref{prop:milman-prop4-2},
we have
\begin{align*}
    D^* = \nabla^{g_{\mathbb{S}^*}, \frac{f_K}{n-1}}.
\end{align*}
Hence,
we have 
\begin{align*}
    \Ric_{g_{\mathbb{S}^*},f_K}^1 = \Ric_{g_{\mathbb{S}^*}}^{D^*}.
\end{align*}
Since the definition of the affine Ricci curvature does not depend on the metric,
we have 
\begin{align*}
    \Ric_{g_{\mathbb{S}^*}}^{D^*} = \Ric_{g_K^{\mathbb{S}^*}}^{D^*}.
\end{align*}
It follows from Proposition \ref{prop:milman-constant} that 
\begin{align*}
    \Ric_{g_K^{\mathbb{S}^*}}^{D^*} = \Ric_{g_K^{\mathbb{S}^*}}^{D} = (n-1)g_{K}^{\mathbb{S}^*}.
\end{align*}
Pulling these identities back by $\mathfrak{n}$ gives the desired results.
\end{proof}
\begin{remark}
For $g$-orthonormal vector fields $X, Y$,
Wylie \cite{wylie2015sectional} defined the weighted sectional curvature $\mathrm{Sec}_{g,f}^1$ as follows (see also Kennard--Wylie--Yeroshkin \cite{kennard2019weighted}):
\begin{align*}
    \mathrm{Sec}_{g,f}^1 (X,Y) := g\left( R^{\nabla^{g, \frac{f}{n-1}}}(Y,X)X,Y \right) = \mathrm{Sec}_g(X,Y) + \frac{1}{n-1}\Hess_g f(X,X) + \frac{(\d f(X))^2}{(n-1)^2},
\end{align*}
where $\mathrm{Sec}_g$ is the sectional curvature of $(M,g)$.
Together with Proposition \ref{prop:milman-prop4-2},
we have 
\begin{align*}
    g_K^{\mathbb{S}^*}(X,X) = \frac{D^2 h_K(X,X)}{h_K} = 1 + \Hess_{g_{\mathbb{S}^*}} \log h_K(X,X) + (X \log h_K)^2 = \mathrm{Sec}_{g_{\mathbb{S}^*},f_K}^1(X,Y)
\end{align*}
for $g_{\mathbb{S}^*}$-orthonormal vector fields $X, Y$.
\end{remark}
\begin{remark}
In Proposition \ref{prop:conn-identity},
the weighted manifold $(\partial K, g_{\partial K}, \mathfrak{m}_{f_{\partial K},g_{\partial K}})$ was considered.
This is the pullback of $(\mathbb{S}^*, g_{\mathbb{S}^*}, \mathfrak{m}_{f_K,g_{\mathbb{S}^*}})$ by $\mathfrak{n}$.
We note that this weighted manifold is different from the weighted manifold considered in Kolesnikov--Milman \cite{kolesnikov2022local} and Milman \cite{milman2023centro}.
\end{remark}
The relation between the setting in this section and the radial graphs is as follows:
\begin{proposition}\label{prop:milman-radial}
Let $K \in \mathcal{K}_+^\infty$.
For $p\in \partial K$,
there exist $A \in GL(\mathbb{R}^{n+1})$,
an open set $U$ in $\partial K$ with $p \in U$,
and a \textcolor{black}{domain} $\Omega$ in $\mathbb{R}^n$ and $0 < h\in C^\infty(\Omega)$ such that 
\begin{align*}
    A(U) = F_h(\Omega).
\end{align*}
Furthermore,
for $\Psi_A := (A|_U)^{-1}: F_h(\Omega) \rightarrow U$,
we have 
\begin{align*}
    \Psi_A^* \left( \nabla_K^{\partial K} \right) = \nabla^{\xi_h},\quad \Psi_A^*\left( g_K^{\partial K} \right) = g^{\xi_h},\quad \Psi_A^* \left( \Ric^{\nabla_K^{\partial K}} \right) = \Ric^{\nabla^{\xi_h}},
\end{align*}
where $\xi_h$ is defined as in Proposition \ref{prop:radial-graph}.
\end{proposition}
\begin{proof}
Since the origin is in the interior of $K$,
there exists $A \in GL(\mathbb{R}^{n+1})$ such that 
\begin{align*}
    A\,p = e_0 := (1,0,\cdots, 0) \in \mathbb{R}^{n+1},\quad A(T_p \partial K) = \{0\} \times \mathbb{R}^n.
\end{align*}
We denote $K_A := A(K)$.
On $\partial K_A \cap \left\{ (y_0,y)\ |\ y\in \mathbb{R}^n, y_0 > 0 \right\}$,
we define 
\begin{align*}
    \Phi((y_0,y)) := \frac{y}{y_0}.
\end{align*}
Then for $(w_0,w) \in T_{e_0} \partial K_A$ with $w_0\in\mathbb{R}$ and $w \in \mathbb{R}^n$,
we have $\d \Phi_{e_0}((w_0,w)) = w$.
Thus 
\begin{align*}
    \d \Phi_{e_0}\big|_{T_{e_0}\partial K_A} : T_{e_0} \partial K_A \rightarrow \mathbb{R}^n
\end{align*}
is an isomorphism.
By the inverse function theorem,
there exists an open set $U_A$ in $\partial K_A \cap \mathbb{R}_{>0} \times \mathbb{R}^n$ with $e_0 \in U_A$ such that $\Phi|_{U_A}$ is a diffeomorphism.
We denote 
\begin{align*}
    \Omega := \Phi (U_A),\quad F := (\Phi|_{U_A})^{-1} : \Omega \rightarrow U_A.
\end{align*}
Also,
we write $F(x) = (F_0(x),F_1(x))$ with $F_0(x) \in \mathbb{R}$ and $F_1(x) \in \mathbb{R}^n$,
and define $h(x) := F_0(x)^{-1}$.
Then we see $F(x) = F_h(x)$.
Consequently,
we have $U_A = F_h(\Omega)$.
Hence,
for $U := A^{-1}(U_A)$,
we have $A(U) = F_h(\Omega)$.
This  yields the first assertion.
The rest follows by applying Proposition \ref{prop:centro-affine-transformation}.
\end{proof}
\subsection{Cheng type maximal diameter theorem for convex bodies}
In this subsection,
we derive the Cheng type maximal diameter property for convex bodies,
and in the subsequent arguments,
we further calculate the diameter of ellipsoids clarifying the role of the assumptions of the Cheng type diameter rigidity.
Our Cheng type maximal diameter theorem is as follows:
\begin{theorem}\label{thm:cheng-milman}
Let $K \in \mathcal{K}_+^\infty$.
For $\kappa > 0$,
we assume 
\begin{align*}
    \Ric_{g_{\partial K}, f_{\partial K}}^1 \geq (n-1) \kappa \,\e^{-\frac{4f_{\partial K}}{n-1}}g_{\partial K}
\end{align*}
and there exists $x,y \in \partial K$ such that 
\begin{align*}
    d_{\widehat{g}_{\partial K}}(x,y) = \frac{\pi}{\sqrt{\kappa}}.
\end{align*}
Then $K$ is a rotationally symmetric ellipsoid.
In particular,
for $e := \mathfrak{n}(x)$ and $\theta \in \mathbb{S}^*$,
we have 
\begin{align}\label{eq:thm:cheng-milman-0}
    h_K (\theta)^2 = h_K(e)^2 \langle \theta, e \rangle^2 + \kappa \,h_K(e)^{-2} |P_{e^{\perp}} \theta|^2.
\end{align}
\end{theorem}
\begin{proof}
We denote $q := \mathfrak{n}(y)$.
The assumptions of this theorem are equivalent to 
\begin{align*}
    \Ric_{g_{\mathbb{S}^*},f_K}^1 \geq (n-1) \kappa\, \e^{-\frac{4f_K}{n-1}} g_{\mathbb{S}^*},\quad d_{\widehat{g}_K^{\mathbb{S}^*}}(e,q) = \frac{\pi}{\sqrt{\kappa}}.
\end{align*}
We first show that we may assume $f_K(e) = 0$ without loss of generality.
We set 
\begin{align*}
    \lambda := h_K(e)^{-1},\quad \widetilde{K} := \lambda K,\quad \widetilde{\kappa} := \lambda^4 \kappa.
\end{align*}
Then we have $h_{\widetilde{K}} = \lambda h_K$ and $f_{\widetilde{K}} = f_K + (n-1)\log \lambda$.
This implies 
\begin{align*}
    \Ric_{g_{\mathbb{S}^*},f_{\widetilde{K}}}^1 = \Ric_{g_{\mathbb{S}^*}, f_K}^1,\quad \widetilde{\kappa} \e^{-\frac{4f_{\widetilde{K}}}{n-1}} = \kappa \e^{-\frac{4f_K}{n-1}}.
\end{align*}
Moreover,
since $\widehat{g}_{\widetilde{K}}^{\mathbb{S}^*} = \lambda^{-4}\widehat{g}_K^{\mathbb{S}^*}$,
we have 
\begin{align*}
    d_{\widehat{g}_{\widetilde{K}}^{\mathbb{S}^*}}(e,q) = \frac{1}{\lambda^2} d_{\widehat{g}_K^{\mathbb{S}^*}}(e,q) = \frac{\pi}{\sqrt{\widetilde{\kappa}}}.
\end{align*}
Therefore,
since $h_{\widetilde{K}}(e) = 1$,
by replacing $(K,\kappa)$ with $(\widetilde{K},\widetilde{\kappa})$,
we may assume $h_K(e) = 1$.

By applying Theorem \ref{thm:cheng} to $(\mathbb{S}^*, g_{\mathbb{S}^*}, f_K)$,
we see $f_K$ depends only on $t := d_{g_{\mathbb{S}^*}}(e,\cdot)$ and the metric is rotationally symmetric.
Hence,
we see $d_{g_{\mathbb{S}^*}}(e,q) = \pi$.
Also from Theorem \ref{thm:cheng},
it follows that
\begin{align*}
    g_{\mathbb{S}^*} = \d t^2 + \e^{\frac{2f_K(\gamma_{e,\theta}(t))}{n-1}}\mathrm{sn}^2_\kappa \left(s_{f_K,e,\theta}(t)\right) g_{\mathbb{S}^{n-1}}
\end{align*}
for $0 < t < d_{g_{\mathbb{S}^*}}(e,q) = \pi$.
On the other hand,
since $g_{\mathbb{S}^*}$ is the standard Riemannian metric on the sphere,
we see
\begin{align*}
    g_{\mathbb{S}^*} = \d t^2 + \sin^2 (t)g_{\mathbb{S}^{n-1}}.
\end{align*}
Comparing these,
we obtain 
\begin{align*}
    \sin^2 t = h_K(\gamma_{e,\theta}(t))^2 \mathrm{sn}_\kappa^2 \left( s_{f_K,e,\theta}(t) \right)
\end{align*}
for $0 < t < \pi$.
This implies 
\begin{align*}
    \sin t = h_K(\gamma_{e,\theta}(t)) \,\mathrm{sn}_\kappa(s_{f_K,e,\theta}(t)).
\end{align*}
It follows that
\begin{align*}
    \frac{\d}{\d t}\mathrm{ct}_\kappa(s_{f_K,e,\theta}(t))
    = -\frac{1}{\mathrm{sn}_\kappa^2\left( s_{f_K,e,\theta}(t)\right) h_K(\gamma_{e,\theta}(t))^2}
    = -\frac{1}{\sin^2 t} = \frac{\d}{\d t} \mathrm{ct}_1 (t).
\end{align*}
Hence,
we obtain 
\begin{align}\label{eq:thm:cheng-milman-2}
    \mathrm{ct}_\kappa(s_{f_K,e,\theta}(t)) = \mathrm{ct}_1(t) + C
\end{align}
for some constant $C$.
We note that since $s_{f_K,e,\theta}(t)$ does not depend on $\theta$,
the constant $C$ does not depend on $\theta$.

We next show that $C = 0$.
Since $h_K$ is radial,
we see $\frac{\d}{\d t} h_K(\gamma_{e,\theta}(t)) |_{t = 0} = 0$.
Hence,
we have
\begin{align*}
    h_K(\gamma_{e,\theta}(t)) = 1 + O(t^2),\quad h_K(\gamma_{e,\theta}(t))^{-2} = 1 + O(t^2).
\end{align*}
This implies
\begin{align*}
    s_{f_K,e,\theta}(t) = t + O(t^3)
\end{align*}
for sufficiently small $t$.
Together with the Taylor expansion of $\mathrm{ct}_\kappa$,
we obtain 
\begin{align*}
    \mathrm{ct}_\kappa (s) = \frac{1}{s} - \frac{\kappa s}{3} + O(s^3).
\end{align*}
This yields
\begin{align*}
    \mathrm{ct}_\kappa(s_{f_K,e,\theta}(t)) - \mathrm{ct}_1 (t) \rightarrow 0
\end{align*}
when $t \searrow 0$.
Combining this with \eqref{eq:thm:cheng-milman-2},
we have $C = 0$.
Therefore,
we have 
\begin{align*}
    \mathrm{ct}_\kappa(s_{f_K,e,\theta}(t)) = \mathrm{ct}_1(t).
\end{align*}
This yields
\begin{align}
    h_K (\gamma_{e,\theta}(t))^2 &= \frac{\sin^2 t}{\mathrm{sn}_\kappa^{2}\left( s_{f_K,e,\theta}(t) \right) }\label{eq:thm-cheng-milman-3}\\
    &= \sin^2 t\ \left( \mathrm{ct}_\kappa (s_{f_K,e,\theta}(t))^2 + \kappa \right)\nonumber\\
    &= \sin^2t\ (\mathrm{cot}^2 t + \kappa)\nonumber \\
    &= \cos^2 t + \kappa \sin^2 t\nonumber
\end{align}
for $0 < t < \pi$.
This is the support function of an ellipsoid of revolution.
Hence,
$K$ is an ellipsoid of revolution.
Applying the argument above to $(\widetilde{K}, \widetilde{\kappa})$ instead of $(K,\kappa)$ yields \eqref{eq:thm:cheng-milman-0}.
\end{proof}
In Theorem \ref{thm:cheng-milman},
the maximal diameter implies that the convex body is a rotationally symmetric ellipsoid.
Conversely,
we calculate the diameter of ellipsoids and see whether they attain the maximal diameter.
Indeed,
we have the following assertion:
\begin{proposition}\label{prop:ellipsoid-milman}
Let $E$ be an $(n+1)$-dimensional vector space over $\mathbb{R}$,
and let $\mathbb{B}$ be the unit ball in $E$.
Also,
let $T : E \rightarrow E$ be a positive-definite symmetric linear map with eigenvalues 
\begin{align*}
    0 < a_1 \leq a_2 \leq \cdots \leq a_{n+1}.
\end{align*}
For the ellipsoid $K := T(\mathbb{B})$,
we have 
\begin{align*}
    \Ric_{g_{\partial K},f_{\partial K}}^1 \geq (n-1)a_1^2 a_2^2 \e^{-\frac{4f_{\partial K}}{n-1}}g_{\partial K}.
\end{align*}
Moreover,
if $a_2 = \cdots = a_{n+1}$,
we have 
\begin{align}\label{eq:prop:ellipsoid-milman-1}
    \mathrm{diam}(\partial K, \widehat{g}_{\partial K}) = \frac{\pi}{a_1 a_2},
\end{align}
and otherwise,
we have 
\begin{align}\label{eq:prop:ellipsoid-milman-2}
    \mathrm{diam}(\partial K, \widehat{g}_{\partial K}) < \frac{\pi}{a_1 a_2}.
\end{align}
\end{proposition}
\begin{proof}
In the argument below,
we identify $E$ with $\mathbb{R}^{n+1}$ and $E^*$.
The following arguments are conducted on $\mathbb{S}^*$.
Pulling back the argument by $\mathfrak{n}$ gives the desired results.
We denote $A := T^2$.
For $\theta \in \mathbb{S}^*$,
we have 
\begin{align}\label{eq:example-ellipsoid-cheng-1}
    h_K (\theta) = h_{T(\mathbb{B})}(\theta) = \max_{x \in \mathbb{B}} \langle Tx,\theta \rangle = |T\theta| = \sqrt{\langle A\theta,\theta \rangle}. 
\end{align}
For $w \in T_\theta \mathbb{S}^*$,
it follows that
\begin{align*}
    \textcolor{black}{D^2 h_K (w,w) = \frac{\langle A w,w\rangle}{h_K(\theta)} - \frac{\langle A\theta,w \rangle^2}{h_K(\theta)^3} }.
\end{align*}
Together with Proposition \ref{prop:milman-prop4-2},
we obtain 
\begin{align*}
    g_K^{\mathbb{S}^*} (w,w) = \frac{\langle A \theta, \theta \rangle \langle Aw,w \rangle - \langle A \theta, w \rangle^2}{h_K(\theta)^4} \geq \frac{a_1^2 a_2^2 |w|^2}{h_K(\theta)^4}.
\end{align*}
Combining this with Proposition \ref{prop:conn-identity},
we see 
\begin{align*}
    \Ric_{g_{\mathbb{S}^*},f_K}^1 = (n-1)g_K^{\mathbb{S}^*} \geq (n-1)a_1^2 a_2^2\, \e^{-\frac{4f_K}{n-1}}g_{\mathbb{S}^*}.
\end{align*}
By Theorem \ref{thm:cheng},
we see 
\begin{align}\label{eq:prop:ellipsoid-milman-3}
    \mathrm{diam}(\partial K, \widehat{g}_{\partial K}) \leq \frac{\pi}{a_1 a_2}.
\end{align}

Next,
we show \eqref{eq:prop:ellipsoid-milman-1}.
For simplicity,
we denote $a := a_1$ and $b := a_2$.
Then there exists $e \in \mathbb{S}^*$ such that $K = T_{e,a,b}(\mathbb{B})$,
where $T_{e,a,b} : E \rightarrow E$ is a map defined as follows:
\begin{align*}
    T_{e,a,b} \, e := a \,e,\quad T_{e,a,b}|_{e^{\perp}} = b \,\mathrm{Id}_{e^{\perp}}.
\end{align*}
For $t(\theta) := d_{g_{\mathbb{S}^*}}(\theta,e)$,
it follows from \eqref{eq:example-ellipsoid-cheng-1} that
\begin{align*}
    h_{K} (\theta)^2 = a^2 \langle \theta, e \rangle^2 + b^2 |P_{e^{\perp}} \theta|^2 = a^2 \cos^2 t(\theta) + b^2 \sin^2 t(\theta).
\end{align*}
Let $\gamma : [0,1] \rightarrow \mathbb{S}^*$ with $\gamma(0) = e$ and $\gamma(1) = -e$.
We denote $\xi(u) := d_{g_{\mathbb{S}^*}}(e, \gamma(u))$.
Then we have 
\begin{align*}
    \left| \frac{\d \gamma}{\d u} (u)\right|_{\widehat{g}_{K}^{\mathbb{S}^*}} 
    \geq \frac{1}{h_K(\gamma(u))^2} \left| \frac{\d \xi}{\d u} (u)\right|.
\end{align*}
Let $L_{\widehat{g}_{K}^{\mathbb{S}^*}} (\gamma)$ denote the length of $\gamma$ with respect to $\widehat{g}_{K}^{\mathbb{S}^*}$.
It follows that
\begin{align*}
    L_{\widehat{g}_{K}^{\mathbb{S}^*}} (\gamma) 
    \geq \int_0^1 \frac{1}{h_K(\gamma(u))^2} \left| \frac{\d \xi}{\d u}(u) \right|\,\d u
    \geq \int_0^\pi \frac{1}{a^2 \cos^2 \xi + b^2 \sin^2 \xi}\ \,\d \xi 
    = \frac{\pi}{ab}.
\end{align*}
Together with \eqref{eq:prop:ellipsoid-milman-3},
this yields \eqref{eq:prop:ellipsoid-milman-1}.

Lastly,
we show \eqref{eq:prop:ellipsoid-milman-2} by contradiction.
We assume the equality holds in \eqref{eq:prop:ellipsoid-milman-3}.
Then there exists $x,y \in \partial K$ such that 
\begin{align*}
    d_{\widehat{g}_{\partial K}} (x,y) = \frac{\pi}{\sqrt{\kappa}},
\end{align*}
where we set $\kappa := a_1^2 a_2^2$.
Below,
we denote $e := \mathfrak{n}(x)$.
Together with \eqref{eq:thm:cheng-milman-0},
we obtain 
\begin{align*}
    \langle T^2 \theta, \theta \rangle = h_K(e)^2 \langle \theta, e\rangle^2 + \kappa \,h_K(e)^{-2} |P_{e^{\perp}}\theta|^2
\end{align*}
for $\theta \in \mathbb{S}^*$.
Hence,
the eigenvalues of $T$ are $h_K(e)$ with multiplicity one and $h_K(e)^{-1} \sqrt{\kappa}$ with multiplicity $n$.
If $h_K(e) \leq h_K(e)^{-1} \sqrt{\kappa}$,
this implies $a_2 = \cdots = a_{n+1}$,
which is a contradiction.
If $h_K(e) > h_K(e)^{-1} \sqrt{\kappa}$,
since $h_K(e)^{-1}\sqrt{\kappa}$ is the eigenvalue with multiplicity $n$ with $n \geq 2$,
we have $a_1 = a_2 = h_K(e)^{-1}\sqrt{\kappa}$.
Hence,
we have 
\begin{align*}
    \kappa = h_K(e)^{4}.
\end{align*}
This contradicts the assumption $h_K(e) > h_K(e)^{-1} \sqrt{\kappa}$.
Therefore,
the equality in \eqref{eq:prop:ellipsoid-milman-3} is impossible.
We arrive at the desired assertion.
\end{proof}

\subsection*{Acknowledgements}

The author is grateful to Professor Yohei Sakurai for valuable comments.
This work was supported by JSPS KAKENHI Grant Number JP25KJ0271.

\subsection*{Use of AI}

The author prepared this manuscript with the aid of ChatGPT.
The author takes full responsibility for the content of this paper.

\bibliographystyle{amsplain}
\bibliography{ref2}
\end{document}